\documentclass[11pt,reqno]{amsart}
\usepackage{amssymb,amsmath,epsfig,mathrsfs, enumerate, xparse, mathtools}
\usepackage[pagewise]{lineno}

\usepackage{graphicx}\usepackage[normalem]{ulem}
\usepackage{fancyhdr}
\allowdisplaybreaks

\usepackage[margin=2.5cm]{geometry}
\usepackage{hyperref}
\hypersetup{
 colorlinks   = true,
 urlcolor     = black,
 linkcolor    = blue,
 citecolor   = red ,
 bookmarksopen=true
}

\def\YYint#1#2#3{{\setbox0=\hbox{$#1{#2#3}{\iint}$}
    \vcenter{\hbox{$#2#3$}}\kern-.50\wd0}}

\def\Xint#1{\mathchoice
	{\XXint\displaystyle\textstyle{#1}}%
	{\XXint\textstyle\scriptstyle{#1}}%
	{\XXint\scriptstyle\scriptscriptstyle{#1}}%
	{\XXint\scriptscriptstyle\scriptscriptstyle{#1}}%
	\!\int}
\def\XXint#1#2#3{{\setbox0=\hbox{$#1{#2#3}{\int}$}
		\vcenter{\hbox{$#2#3$}}\kern-.5\wd0}}
\def\aver#1{\Xint-_{#1}}

\def\Xint#1{\mathchoice
    {\XXint\displaystyle\textstyle{#1}}%
    {\XXint\textstyle\scriptstyle{#1}}%
    {\XXint\scriptstyle\scriptscriptstyle{#1}}%
    {\XXint\scriptscriptstyle\scriptscriptstyle{#1}}%
      \!\int}
\def\XXint#1#2#3{{\setbox0=\hbox{$#1{#2#3}{\int}$}
    \vcenter{\hbox{$#2#3$}}\kern-.50\wd0}}

 \usepackage{tikz}
 \newcommand{\mb}{\mathbb}
 \newcommand{\mc}{\mathcal}
 
 \newcommand{\f}{\frac}
 \newcommand{\ld}{\lambda}

 \newcommand{\dd}{\partial}

 \newcommand{\iy}{\infty}

 \newcommand{\tn}{\textnormal}
 \newcommand{\py}{  \partial_{x_{n+1}}^a} 
 \usepackage[notref,notcite]{}
 \usepackage[pagewise]{lineno}
 \newtheorem{theorem}{Theorem}[section]
 \newtheorem{lemma}[theorem]{Lemma}
 
 \newtheorem{corollary}[theorem]{Corollary}
 
 \newtheorem{remark}[theorem]{Remark}

 \numberwithin{equation}{section}

 \usepackage{hyperref}

 \newcommand{\Rn}{\mathbb{R}^n}

 \newcommand{\bp}{\begin{prob}}
 	\newcommand{\bpr}{\begin{proof}}
 		\newcommand{\epr}{\end{proof}}

 	\newcommand{\Rb}{\mathbb{R}}

		\usepackage{amsmath}
 	\theoremstyle{definition}

 	\title[Decay at infinity etc.]{Decay at infinity for solutions to some fractional parabolic equations}

  \author{Agnid Banerjee}
\address{TIFR CAM, Bangalore-560065} \email[Agnid Banerjee]{agnidban@gmail.com}

\author{Abhishek Ghosh  }
\address{TIFR CAM, Bangalore-560065} 
\curraddr{Institute of Mathematics, Polish Academy of Sciences, Ul. \'Sniadeckich 8, 00-656 Warsaw, Poland.}
\email[Abhishek Ghosh]{abhi170791@gmail.com; agosh@impan.pl}
\thanks{A.B is supported in part  by Department of Atomic Energy,  Government of India, under
project no.  12-R \& D-TFR-5.01-0520. A. G is supported by TIFR-CAM, Bangalore-560065, India.}


\begin{document}
		\begin{abstract}
		For $s \in [1/2, 1)$, let $u$ solve $(\partial_t - \Delta)^s u = Vu$ in $\mathbb R^{n}  \times [-T, 0]$ for some $T>0$ where  $||V||_{ C^2(\mathbb R^n \times [-T, 0])} < \infty$. We show that     if for some $0<c< T$ and $\epsilon>0$  $$\aver{[-c,0]} u^2(x, t) dt \leq Ce^{-|x|^{2+\epsilon}}\  \forall x \in \mathbb R^n,$$ then $u \equiv 0$ in $\mathbb R^{n} \times [-T, 0]$.  
	\end{abstract}
	
  \maketitle
  
  \tableofcontents

 	
 	%
\section{Introduction}
E.M Landis and O. A. Oleinik in \cite{LO} asked the following question:

\medskip

\textbf{Question A:} \emph{Let $u$ be a bounded solution to the following parabolic differential inequality
\begin{equation}\label{locpar}
|\Delta u  - u_t| \leq C( |u| + |\nabla u|)
\end{equation}
in $\mathbb R^n \times [-T, 0]$ such that for some $\epsilon>0$ \begin{equation}\label{decay} |u(x,0)| \leq Ce^{-|x|^{2+\epsilon}}, \forall x \in \mathbb R^n. \end{equation} Then is  $u \equiv 0$ in $\mathbb R^{n} \times [-T, 0]$? }

This question was answered in affirmative  in the work \cite{EKVP} where among other things, the authors  showed   that the following decay estimate at infinity holds for bounded  solutions to \eqref{locpar} provided $$||u(\cdot, 0)||_{L^{2}(B_1)} > 0.$$
\begin{equation}\label{ekvp1}
||u(\cdot, 0)||_{L^{2}(B_1(x))} \geq e^{-N |x|^2 \log |x|},\ |x| \geq N\ \text{where $N$ is some large universal constant}.
\end{equation}
Now using the estimate \eqref{ekvp1}, the answer to the Landis-Oleinik conjecture is seen as follows:

\emph{Assume that the decay as in \eqref{decay} holds.  Since \[ e^{- N R^2 \log R} \gg e^{-R^{2+\epsilon}},\] as $R \to \infty$, thus   from  \eqref{ekvp1} it follows that $u(\cdot, 0) \equiv 0$ in $B_1$. Now by applying the  space like strong unique continuation result in \cite{EF_2003, EFV_2006} we deduce that $u(\cdot, 0) \equiv 0$ in $\mathbb R^n$. Subsequently by applying the backward uniqueness result in \cite{Ch, ESS, Po} we find that $u \equiv 0$ in $\mathbb R^n \times (-T, 0]$.}

We also refer to \cite[Theorem 4]{EKVP1} for a related result.  See also \cite{EKPV16} for a further sharpening of the result in \cite{EKVP1}. The proof of the inequality  \eqref{ekvp1} in \cite{EKVP} is based on a fairly non trivial application of a Carleman estimate derived in the pioneering work of Escauriaza-Fernandez-Vessella in \cite{EF_2003, EFV_2006}   on space like strong unique continuation for local parabolic equations combined with an appropriate rescaling argument inspired by ideas in \cite{BK}.  Over here, we would like to mention that such results are of interest in control theory, see for instance \cite{MZ}.  They have also turned out to be useful in the regularity theory for Navier Stokes equations, see \cite{SS}. 

\subsection{Statement of the main results}
In this work, we derive the following nonlocal analogue of the estimate in \eqref{ekvp1}. We refer to Section \ref{s:n} for the relevant notions and notations. 

\begin{theorem}\label{main}
For $s \in [1/2, 1)$, let $u \in \text{Dom}(H^s)$ be a solution to 
\begin{equation}\label{e0}
(\partial_t - \Delta)^s u = Vu,
\end{equation}
in $\mathbb R^n \times [-T, 0]$ where  $||V||_{ C^{2}(\mathbb R^n \times [-T, 0])} \leq C$. Assume  that for some $0< c<T$ \begin{equation}\label{non1} ||u||_{L^{2}\left(B_{\frac{\sqrt{c}}{2}} \times (-c/4, 0]\right)} \geq \theta>0. \end{equation} Then there exists universal $M>1$ large enough depending on $\theta, s, n, c$ and $C$ such that $\forall x_0 \in \mathbb R^n$ with  $|x_0| \geq M$, we have
\begin{equation}\label{main1}
\int_{B_2(x_0) \times (-c, 0]} u^2 dxdt > e^{-M |x_0|^2 \log |x_0|}.
\end{equation}

\end{theorem}
 
As a consequence of Theorem \ref{main}, the following ``average  in time''  version of  the Landis-Oleinik type result follows in our nonlocal setting.

\begin{corollary}\label{main0}
For $s \in [1/2, 1)$, let $u \in \text{Dom}(H^s)$ be a  solution to \eqref{e0} in $\mathbb R^n \times [-T, 0]$. If for some $\epsilon >0$ and  $0< c<T$, we have that  $$\aver{[-c,0]} u^2 (x,t)  dt \leq Ce^{-|x|^{2+\epsilon}},\  \forall x \in \mathbb R^{n},$$ then $u \equiv 0$ in $\mathbb R^n \times [-T, 0]$. 

\end{corollary}

\subsection{Key ideas in the proof of Theorem \ref{main}:}
The following are the key steps in the proof of our main result Theorem \ref{main}.

\emph{Step 1:} Via a compactness argument as in Lemma \ref{com-un} with a monotonicity in time result in \cite[Lemma 3.1]{ABDG}, we first show that a non-degeneracy condition  at the boundary for the nonlocal problem as in \eqref{non1} implies a   similar non-vanishing  condition for the corresponding extension problem \eqref{exprob}. See Lemma \ref{com-conse1} below.

\emph{Step 2:} Then by means of a quantitative monotonicity in time result  as in Lemma \ref{lemma-mont} and a quantitative Carleman type estimate  as in Theorem \ref{carl1-un}, we show by adapting the rescaling  arguments in \cite{EKVP}  that the solution $U$ to the corresponding extension problem satisfies a similar decay estimate at infinity as in \eqref{main1} above.  See Theorem \ref{main-thick} below. We would like to mention that both Lemma \ref{lemma-mont} and the Carleman estimate in Theorem \ref{carl1-un} are subtle variants of the estimates recently established by two of us in \cite{AA}.  The  main new feature of the both the results  is a certain quantitative dependence of the estimates   on the  rescaling parameter $R$  ( see \eqref{u_r} below)  as  $R \to \infty$. This is precisely where we require $s \geq 1/2$. This constitutes the key novelty of our work.

\emph{Step 3:} The decay estimate at infinity for the extension problem is then transferred to  the non-local problem by using a propagation of smallness estimate derived in \cite{AryaBan}. Such a propagation of smallness estimate constitutes the parabolic analogue of the one due to Ruland and Salo in \cite{RS}.  It is to be noted that via the propagation of smallness  estimate  in \eqref{pos} below, the  transfer  of  the decay information from the bulk in the extension problem \eqref{exprob} to the boundary in the nonlocal problem \eqref{e0}   occurs only in ``space-time'' regions and not at a given time level. This is precisely why  we require  an ``average in time''  decay assumption   in Corollary \ref{main0} instead of a pointwise decay assumption at $t=0$ for the nonlocal Landis-Oleinik type result to hold.

For various results on unique continuation for nonlocal fractional laplacian type equations and its time dependent counterpart, we refer to \cite{AryaBan, ABDG, AT, BG, BG1,  BVS, BS1, AA,   BL,  FF, FPS,  LLR,  Ru, Ru1, Ru2,  RS, RSV, RW, Yu, Zhu0}, each of which are either based on Carleman estimates  as in \cite{AKS}  or  on the  frequency function approach  as in \cite{GL} followed by a blowup argument.

The paper is organized as follows. In Section \ref{s:n}, we introduce some basic notations and notions and gather some known results that are relevant for our work. In Section \ref{s:3}, we prove our key estimates in Lemma \ref{lemma-mont} and Theorem \ref{carl1-un}. In Section \ref{s:main}, we finally prove our main results Theorem \ref{main} and Corollary \ref{main0}.

\section{Preliminaries}\label{s:n}
In this section we introduce the relevant notation and gather some auxiliary  results that will be useful in the rest of the paper. Generic points in $\mathbb R^n \times \mathbb R$ will be denoted by $(x_0, t_0), (x,t)$, etc. For an open set $\Omega\subset \mathbb R^n_x\times \mathbb R_t$ we indicate with $C_0^{\infty}(\Omega)$ the set of compactly supported smooth functions in $\Omega$.  We also indicate by $H^{\alpha}(\Omega)$ the non-isotropic parabolic H\"older space with exponent $\alpha$ defined in \cite[p. 46]{Li}. The symbol $\mathscr S(\mathbb R^{n+1})$ will denote the Schwartz space of rapidly decreasing functions in $\mathbb R^{n+1}$. For $f\in \mathscr S(\mathbb R^{n+1})$ we denote its Fourier transform by 
\[
 \hat f(\xi,\sigma) = \int_{\mathbb R^n\times \mathbb R} e^{-2\pi i(\langle \xi,x\rangle + \sigma t)} f(x,t) dx dt = \mathscr F_{x\to\xi}(\mathscr F_{t\to\sigma} f).
\]
The heat operator in $\mathbb R^{n+1} = \mathbb R^n_x \times \mathbb R_t$ will be denoted by $H = \partial_t - \Delta_x$. Given a number $s\in (0,1)$ the notation $H^s$ will indicate the fractional power of $H$ that in  \cite[formula (2.1)]{Samko} was defined on a function $f\in \mathscr S(\mathbb R^{n+1})$ by the formula
\begin{equation}\label{sHft}
\widehat{H^s f}(\xi,\sigma) = (4\pi^2 |\xi|^2 + 2\pi i \sigma)^s\  \hat f(\xi,\sigma),
\end{equation}
where we have chosen the principal branch of the complex function $z\to z^s$. Consequently, we have that the natural domain of definition of $H^s$ is as follows:
\begin{align}\label{dom}
\mathscr H^{2s} & =  \operatorname{Dom}(H^s)   = \{f\in \mathscr S'(\mathbb R^{n+1})\mid f, H^s f \in L^2(\mathbb R^{n+1})\}
\\
&  = \{f\in L^2(\mathbb R^{n+1})\mid (\xi,\sigma) \to (4\pi^2 |\xi|^2 + 2\pi i \sigma)^s  \hat f(\xi,\sigma)\in L^2(\mathbb R^{n+1})\},
\notag
\end{align} 
where the second equality is justified by \eqref{sHft} and Plancherel theorem. 
Such a definition via the Fourier transform  is equivalent to the one based on Balakrishnan formula (see \cite[(9.63) on p. 285]{Samko})
\begin{equation}\label{balah}
H^s f(x,t) = - \frac{s}{\Gamma(1-s)} \int_0^\infty \frac{1}{\tau^{1+s}} \big(P^H_\tau f(x,t) - f(x,t)\big) d\tau,
\end{equation}
where 
\begin{equation}\label{evolutivesemi}
P^H_\tau f(x,t) = \int_{\Rn} G(x-y,\tau) f(y,t-\tau) dy = G(\cdot,\tau) \star f(\cdot,t-\tau)(x)
\end{equation}
the \emph{evolutive semigroup}, see \cite[(9.58) on p. 284]{Samko}. We refer to Section 3 in \cite{BG} for relevant details.

Henceforth, given a point $(x,t)\in \mathbb R^{n+1}$ we will consider the thick half-space $\mathbb R^{n+1} \times \mathbb R^+_{x_{n+1}}$. At times it will be convenient to combine the additional variable $x_{n+1}>0$ with $x\in \Rn$ and denote the generic point in the thick space $\Rn_x\times\mathbb R^+_{x_{n+1}} := \mathbb R^{n+1}_+$ with the letter $X=(x,x_{n+1})$. For $x_0\in \mathbb R^n$ and $r>0$ we let $B_r(x_0) = \{x\in \Rn\mid |x-x_0|<r\}$,
$\mathbb B_r(X)=\{Z = (z,z_{n+1}) \in \mathbb R^n \times \mathbb R \mid |x-z|^2 + |x_{n+1}- z_{n+1}|^2 < r^2\}$. We also let $\mathbb B_r^+(X)= \mathbb B_r(X) \cap \{(z, z_{n+1}: z_{n+1} >0\}$. 
When the center $x_0$ of $B_r(x_0)$ is not explicitly indicated, then we are taking $x_0 = 0$. Similar agreement for the thick half-balls $\mathbb B_r^+((x_0,0))$. We will also use the $\mathbb Q_{r}$ for the set $\mathbb B_r \times [t_0,t_0+r^2)$ and $Q_r$ for the set  $ B_r \times [t_0,t_0+r^2).$ Likewise we denote  $\mathbb Q_r^+=\mathbb Q_r \cap \{(x,x_{n+1}): x_{n+1} > 0\}$. 
For notational ease $\nabla U$ and  $\operatorname{div} U$ will respectively refer to the quantities  $\nabla_X U$ and $ \operatorname{div}_X U$.  The partial derivative in $t$ will be denoted by $\partial_t U$ and also at times  by $U_t$. The partial derivative $\partial_{x_i} U$  will be denoted by $U_i$. At times,  the partial derivative $\partial_{x_{n+1}} U$  will be denoted by $U_{n+1}$.

 We next introduce the extension problem associated with $H^s$.  
 Given a number $a\in (-1,1)$ and a $u:\mathbb R^n_x\times \mathbb R_t\to \mathbb R$ we seek a function $U:\mathbb R^n_x\times\mathbb R_t\times \mathbb R_{x_{n+1}}^+\to \mathbb R$ that satisfies the boundary-value problem
\begin{equation}\label{la}
\begin{cases}
\mathscr{L}_a U \overset{def}{=} \partial_t (x_{n+1}^a U) - \operatorname{div} (x_{n+1}^a \nabla U) = 0,
\\
U((x,t),0) = u(x,t),\ \ \ \ \ \ \ \ \ \ \ (x,t)\in \mathbb R^{n+1}.
\end{cases}
\end{equation}
The most basic property of the Dirichlet problem \eqref{la} is that if \begin{equation}\label{sa} s = \frac{1-a}2\in (0,1) \end{equation} and $u \in \text{Dom}(H^{s})$, then we have the following convergence  in $L^{2}(\mathbb R^{n+1})$
\begin{equation}\label{np}
2^{-a}\frac{\Gamma(\frac{1-a}{2})}{\Gamma(\frac{1+a}{2})} \py U((x,t),0)=  - H^s u(x,t),
\end{equation}
where $\py$ denotes the weighted normal derivative
\begin{equation}\label{nder}
\py U((x,t),0)\overset{def}{=}   \lim\limits_{x_{n+1} \to 0^+}  x_{n+1}^a \partial_{x_{n+1}} U((x,t),x_{n+1}).
\end{equation}

When $a = 0$ ($s = 1/2$) the problem \eqref{la} was first introduced in \cite{Jr1} by Frank Jones, who in such case also constructed the relevant Poisson kernel and proved \eqref{np}. More recently Nystr\"om and Sande in \cite{NS} and Stinga and Torrea in \cite{ST} have independently extended the results in \cite{Jr1} to all $a\in (-1,1)$. 	

With this being said, we now suppose that $u$ be a solution to \eqref{e0}  and consider the weak solution $U$ of the following version of \eqref{la} (for the precise notion of weak solution of \eqref{wk} we refer to \cite[Section 4]{BG}) 
\begin{equation}\label{wk}
\begin{cases}
\mathscr{L}_a U=0 \ \ \ \ \ \ \ \ \ \ \ \ \ \ \ \ \ \ \ \ \ \ \ \ \ \ \ \ \ \ \text{in}\ \mathbb R^{n+1}\times \mathbb R^+_{x_{n+1}},
\\
U((x,t),0)= u(x,t)\ \ \ \ \ \ \ \ \ \ \ \ \ \ \ \ \text{for}\ (x,t)\in \mathbb R^{n+1},
\\
\py U((x,t),0)=  2^{a} \frac{\Gamma(\frac{1+a}{2})}{\Gamma(\frac{1-a}{2})} V(x,t) u(x,t)\ \ \ \ \text{for}\ (x,t)\in \mathbb R^n \times (-T,0].
\end{cases}
\end{equation}
To simplify notation, we will let $2^{a} \frac{\Gamma(\frac{1+a}{2})}{\Gamma(\frac{1-a}{2})} V(x,t)$ as our new $V(x,t)$. Note that the third equation in \eqref{wk} is justified by \eqref{e0} and \eqref{np}.  From now on, a generic point $((x,t), y)$ will be denoted as $(X,t)$ with $X=(x,y)$.  Further,  as in \cite[Lemma 5.3]{BG}  ( see also \cite[Lemma 2.2]{AryaBan}),  the following regularity result for such weak solutions was proved. Such result will be relevant to our analysis. For simiplicity, we assume that $T>4$. We refer to \cite[Chapter 4]{Li} for the relevant notion parabolic H\"older spaces.

\begin{lemma}\label{reg1}
Let $U$  be a weak solution of \eqref{wk} where $V \in C^2(\mathbb R^n \times (-T, 0])$. Then there exists $\alpha'>0$ such that one has up to the thin set $\{x_{n+1}=0\}$ 
\[
U_i,\ U_t,\ x_{n+1}^a U_{x_{n+1}}\ \in\ H^{\alpha'}.
\]
Moreover, the relevant H\"older norms  over a compact set $K$ are bounded by $\int U^2 x_{n+1}^a dX dt$ over a larger set $K'$ which contains $K$. We also  have that $\nabla_x^2 U \in C^{\alpha'}_{loc}$ up to the thin  set $\{x_{n+1}=0\}$. Furthermore, we have that the following estimate holds for $i, j=1, ..,n$ and $x_0 \in \mathbb R^n$
\begin{align}\label{ret}
&\int_{\mathbb B_1^+ ((x_0, 0)) \times (-1, 0]} (U_t^2+ U_{tt}^2) x_{n+1}^a +\int_{\mathbb B_{2}^+ ((x_0, 0)) \times (-4, 0]} |\nabla U_{t}|^2 x_{n+1}^a  +  \int_{\mathbb B_2^+((x_0, 0)) \times (-4, 0]} |\nabla U_{ij}|^2 x_{n+1}^a \\&\leq C( 1+ ||V||_{C^2}) \int_{\mathbb B_2^+((x_0, 0)) \times (-4, 0]} U^2 x_{n+1}^a, \notag\end{align}
where $C$ is some universal constant.

\end{lemma}  
We also record the following result as in \cite[Corollary 5.3]{BG} that will be needed in our work.

\begin{lemma}\label{reg}
Let $U$ be as in \eqref{wk}. Then we have that $||U||_{L^{\infty}(\mathbb R^{n+1}_+)} \leq C$ for some universal $C$ depending on $||u||_{\mathscr H^{2s} (\mathbb R^{n+1})}$ and $||V||_{C^2}$. 
\end{lemma}

For notational purposes it will be convenient to work with the following backward version of problem \eqref{wk}.	
	
	\begin{equation}\label{exprob}
		\begin{cases}
			x_{n+1}^a \partial_t U + \operatorname{div}(x_{n+1}^a \nabla U)=0\ \ \ \ \ \ \ \ \ \ \ \ \text{in} \ \mathbb R^{n+1}_+ \times [0, T),
			\\	
			U(x,0, t)= u(x,t)
			\\
			\py U(x, 0,t)= Vu\ \ \ \ \ \ \ \ \ \text{in}\ \mathbb R^n \times [0,T).
		\end{cases}
	\end{equation}
	
	We note that the former can be transformed into the latter by changing $t \to -t$.  
	
The corresponding extended backward parabolic operator will be denoted as
	\begin{align}\label{extop}
		\widetilde{\mc{H}}_s & := x_{n+1}^a \dd_t + \tn{div} \left( x_{n+1}^a \nabla \right).
	\end{align}

We now collect some auxiliary results that will be needed in the proof of our main Carleman estimate in Theorem \ref{carl1-un}.
\begin{lemma}[Lemma 2.3 in \cite{AA}, \cite{EF_2003}] \label{sigma}
	Let $s\in (0, 1).$ Define
		\begin{equation}\label{theta'}\theta_{s}(t) = t^{s} \left( \log \f{1}{t} \right)^{1+s}.\end{equation}  Then the solution to the ordinary differential equation 
		$$\frac{d}{dt} \log \left(\frac{\sigma_{s}}{t\sigma_{s}'}\right)= \frac{\theta_{s}(\lambda t)}{t},~\sigma_{s}(0)=0,~\sigma_{s}'(0)=1,$$
		where $\lambda >0,$ has the following properties when $0\leq \lambda t\leq 1$:
		\begin{enumerate}
			\item $t e^{-N} \leq \sigma_{s}(t) \leq t,$
			\item $e^{-N} \leq \sigma_{s}'(t)\leq 1,$
			\item $|\partial_t[\sigma_{s} \log \frac{\sigma_{s}}{\sigma_{s}' t}]|+|\partial_t[\sigma_{s} \log \frac{\sigma_{s}}{\sigma_{s}' }]|\leq 3N$,
			\item $\left|\sigma_{s} \partial_t \left(\frac{1}{\sigma_{s}'}\partial_t[\log \frac{\sigma_{s}}{\sigma_{s}'(t)t}]\right)\right| \leq 3N e^{N} \frac{\theta_{s}(\gamma t)}{t},$
		\end{enumerate}	
		where $N$ is some universal constant. 
	\end{lemma}

\begin{lemma}[Trace inequality]\label{tr}
Let   $f\in C_0^\infty(\overline{\mathbb R^{n+1}_+})$. There exists a constant $C_0 = C_0(n,a)>0$ such that for every $A>1$ one has
\[
\int_{\Rn} f(x,0)^2 dx \le C_0 \left(A^{1+a} \int_{\mathbb R^{n+1}_+} f(X)^2 x_{n+1}^a dX + A^{a-1} \int_{\mathbb  R^{n+1}_+} |\nabla f(X)|^2 x_{n+1}^a dX\right).
\]
\end{lemma}
 
\begin{lemma}\label{do}
		Assume that $N \ge 1,$ $h \in C_{0}^{\infty}(\overline{\mathbb R^{n+1}_+})$ and the inequality \begin{align*}
			2b \int_{\mathbb R^{n+1}_+} x_{n+1}^a|\nabla h|^2 e^{-|X|^2/4b}dX + \frac{n+1+a}{2}\int_{\mathbb{R}^{n+1}_+}x_{n+1}^a h^2 e^{-|X|^2/4b}dX \le N  \int_{\mathbb{R}^{n+1}_+} x_{n+1}^ah^2  e^{-|X|^2/4b}dX
		\end{align*}
		holds for $b \le \frac{1}{12N }.$ Then
		\begin{align}
			\int_{\mathbb B_{2r}^+}h^2 x_{n+1}^a dX \le e^N  \int_{\mathbb B_{r}^+}h^2 x_{n+1}^a dX 
		\end{align}
		when $0 < r \le 1/2.$
	\end{lemma}
	We also need the following Hardy type inequality in the Gaussian space which  can be found in Lemma 2.2 in \cite{ABDG}.  This can be regarded as the weighted analogue of  Lemma 3 in \cite{EFV_2006}.
	\begin{lemma}[Hardy type inequality]\label{hardy}
		For all $h \in C_0^{\infty}(\overline{\mb{R}^{n+1}_+})$ and $b>0$ the following inequality holds
		\begin{align*}
			& \int_{\mb{R}^{n+1}_+} x_{n+1}^a h^2 \frac{|X|^2}{8b} e^{-|X|^2/4b}dX \leq  2b \int_{\mb{R}^{n+1}_+} x_{n+1}^a|\nabla h|^2 e^{-|X|^2/4b} dX
			\\
			& + \frac{n+1+a}{2} \int_{\mb{R}^{n+1}_+}x_{n+1}^a h^2  e^{-|X|^2/4b}  dX.
		\end{align*}
	\end{lemma}
	
	Finally, we also need the following interpolation type inequality as in \cite[Lemma 2.4]{AryaBan}.
	
	\begin{lemma}\label{interpol}
	Let $s \in (0,1)$ and $f \in C^{2}_0(\mathbb R^n \times \mathbb R_+)$. Then there exists a  universal constant $C$ such that for any $0<\eta<1$ the following holds
	\begin{equation}\label{inter1}
	||\nabla_x f||_{L^2(\mathbb R^n \times \{0\})} \leq C \eta^s \left(||x_{n+1}^{a/2} \nabla \nabla_x f||_{L^2(\mathbb R^n \times \mathbb R_+)} + ||x_{n+1}^{a/2} \nabla_x f||_{L^2(\mathbb R^n \times \mathbb R_+)} \right) + C\eta^{-1} ||f||_{L^2(\mathbb R^n \times \{0\})}.\end{equation}
	In particular when $n=1$, we get
	\begin{equation}\label{inter2}
|| f_t||_{L^{2}(\mathbb R \times \{0\})} \leq C \eta^s\left( || x_{n+1}^{a/2} \partial_{x_{n+1}} f_t||_{L^2(\mathbb R \times \mathbb R_+)}  +||x_{n+1}^{a/2}f_{tt}||_{L^2(\mathbb R \times \mathbb R_+)}  +||x_{n+1}^{a/2}f_{t}||_{L^2(\mathbb R \times \mathbb R_+)}\right) + C\eta^{-1} ||f||_{L^2(\mathbb R \times \{0\})}.	\end{equation}

\end{lemma}
 
\section{The key lemmas}\label{s:3}

For the simplicity of exposition, we will assume that $c=1$ in Theorem \ref{main} and Corollary \ref{main0}. Furthermore, we make the following remark.

\begin{remark}
In the rest of the paper, the constant $c$ that appears in Lemma \ref{lemma-mont}, Theorem \ref{carl1-un} and Theorem \ref{estimates-for-resc} is different from   the one in Theorem \ref{main} and Corollary \ref{main0}.
\end{remark}

We will also  assume that
\begin{equation}\label{vasump}
||V||_{C^2_{(x,t)} (\Rn \times (-T, 0))} \leq 1.
\end{equation}

We first show that via a compactness argument,  the non-vanishing condition at the boundary for the nonlocal problem \eqref{e0} as in \eqref{non1} implies a similar non-vanishing for the extension problem \eqref{exprob}.  Since the proof is via compactness, we show this result for a larger ``compact'' family of solutions to \eqref{exprob}.
\begin{lemma}[Bulk non-degeneracy]
\label{com-un}
Let  $W$ be a solution to 

\begin{equation}\label{exprob1}
		\begin{cases}
			x_{n+1}^a \partial_t W + \operatorname{div}(x_{n+1}^a \nabla W)=0\ \ \ \ \ \ \ \ \ \ \ \ \text{in} \ \ \mathbb{R}_{+}^{n+1} \times [0, 25),
			\\	
			\py W(x, 0,t)= \tilde VW\ \ \ \ \ \ \ \ \ \text{in}\ B_5 \times [0,25),
		\end{cases}
	\end{equation}
where $\tilde V$ satisfies \eqref{vasump}.
Furthermore, assume that 
$\|W\|_{L^{\infty}(\mathbb{Q}_5^+)}\leq C$ and  $\int_{Q_{1/2}}W^2(x, 0, t) dx\,dt\geq \theta>0.$ Then there exists a constant $\kappa:=\kappa(\theta, a, n)>0$ such that \begin{equation}\label{nonden}\int_{\mathbb Q_{1/2}^{+}} x_{n+1}^{a}W^2 dX\, dt\geq \kappa.\end{equation}
\end{lemma}
\begin{proof}
On the contrary if there does not exists any $\kappa,$ then for each $j\in \mathbb{N}$ there exists $W_{j}$ such that $\int_{Q_{1/2}}W_{j}^2(x, 0, t) dx\,dt\geq \theta$,
\begin{align}
\label{compact-un}
\int_{\mathbb Q_{1/2}^{+}} x_{n+1}^{a}W_{j}^2 dX\, dt<\frac{1}{j},\end{align}
and
\begin{equation}\label{bounded}
||W_{j}||_{L^{\infty}(\mathbb Q_5^+)}\leq C.
\end{equation}
Moreover, $W_j$ solves the problem 
\begin{equation}\label{exprob-com}
		\begin{cases}
			x_{n+1}^a \partial_t W_j + \operatorname{div}(x_{n+1}^a \nabla W_j)=0\ \ \ \ \ \ \ \ \ \ \ \ \text{in} \ \ \mathbb{Q}_{5}^{+}\\
			\py W_j(x, 0,t)= V_j W_j\ \ \ \ \ \ \ \ \ \text{in}\ Q_5,
		\end{cases}
	\end{equation}
	with $V_j$'s satisfying the bound in \eqref{vasump}.

The key thing is to notice that, from the  regularity estimates in Lemma~\ref{reg1} and \eqref{bounded}, the H\"older norms of $W_j's$ are uniformly bounded. So using Arzel\'a-Ascoli, possibly passing through a subsequence, $W_j\to W_0$ in $H^{\alpha}(\mathbb{Q}_{2}^{+})$ up to $\{x_{n+1}=0\}$ for some $\alpha>0$.  Consequently, using \eqref{compact-un} and uniform convergence, we have
\begin{align}
\label{compact-2-un}
\int_{\mathbb Q_{1/2}^{+}} x_{n+1}^{a}W_{0}^2 dX\, dt=0.
\end{align}
Again $\int_{Q_{1/2}}W_{j}(x,0,t)^2 dx\,dt\geq\theta$ implies by uniform convergence that  $\int_{Q_{1/2}}W_{0}(x, 0, t)^2 dx\,dt\geq \theta>0.$ This contradicts \eqref{compact-2-un} and thus the conclusion follows.
\end{proof}
We now record the following important  consequence of Lemma~\ref{com-un}.

\begin{lemma}
\label{com-conse}
With $U$ as in \eqref{exprob} there exists $\gamma >0$ and  some $t_{0}\in [0, 1/4-\gamma)$ such that
\begin{align}
\label{lower0}
 \int_{\mathbb B_{1/2}^{+}} x_{n+1}^{a}U^2(X, t_{0})\, dX\geq \kappa.
\end{align}
\end{lemma}
\begin{proof}
We choose $t_0$ as  \begin{equation}\label{t0} t_0= \inf \bigg\{t \in (0,1/4): \int_{\mathbb B_{1/2}^+} x_{n+1}^a U^2(X, t) \,dX \geq \kappa \bigg\}.\end{equation} Thanks to \eqref{nonden} (which also applies to $U$ as \eqref{non1} holds), the corresponding set is non-empty and thus $t_0$ exists. The existence of $\gamma$ follows from the fact that from \eqref{nonden}, Lemma \ref{reg1} and the definition of $t_0$ as in \eqref{t0},  we have
\begin{equation}\label{in}
\kappa \leq \int_{\mathbb Q_{1/2}^+} x_{n+1}^a U^2 = \int_{0}^{t_0} \int_{\mathbb B_{1/2}^+}  x_{n+1}^a U^2+ \int_{t_0}^{1/4} \int_{\mathbb B_{1/2}^+} x_{n+1}^a U^2   \leq \kappa t_0 + (1/4 -t_0) \tilde C
\end{equation}
where $\tilde C= C^2 \int_{\mathbb B_{1/2}^+} x_{n+1}^adX,$ with $C$  as in Lemma \ref{reg}. From \eqref{in}  we find using $t_0 \leq 1/4$ that the following inequality holds
\[
\kappa \leq \f{\kappa}{4} + (1/4-t_0) \tilde C,
\]
which in turn implies that
\begin{equation}\label{k1}
(1/4- t_0) \geq \f{3\kappa}{4\tilde C}\ \ .
\end{equation}
Therefore, $\gamma$ can be taken as  $\f{3\kappa}{4\tilde C}$ which implies the desired conclusion.

\end{proof}

Lemma \ref{com-conse} combined with the monotonicity in time result in \cite[Lemma 3.1]{ABDG} implies the following non-degeneracy estimate for $U$ in space-time.

\begin{lemma}\label{com-conse1}
With the assumptions as in Lemma \ref{com-conse} above, we have that there exists $0<\tilde \delta < \gamma$ ($\gamma$ as in Lemma \ref{com-conse} above) and $\tilde \kappa \in (0,1)$ such that for $\tilde t \in [t_0, t_0+\tilde \delta)$, we have
\begin{align}
\label{lower}
\int_{\mathbb B_{1}^+} x_{n+1}^a U^2(X, \tilde t) dX \geq  \tilde \kappa.
\end{align}
\end{lemma}

\begin{proof}
First, we note that from Lemma \ref{com-conse},  there exists $\gamma>0$ and  $t_{0}\in [0, 1-\gamma)$  such that
\begin{align}
\label{lower1}
 \int_{\mathbb B_{1/2}^{+}} x_{n+1}^{a}U^2(X, t_{0}) dX\geq \kappa.
\end{align}
Then by applying the monotonicity result in \cite[Lemma 3.1]{ABDG}, we have that  for $c_0, c_1 \in (0,1)$   depending  on $n,s, \kappa$ and $C$ in Lemma \ref{reg},  the following inequality holds for  all $t \in [t_0, t_0 +c_0)$
\begin{align}
\label{lower2}
 \int_{\mathbb B_{1}^{+}} x_{n+1}^{a}U^2(X, t) dX\geq c_1 \kappa.
\end{align}
We now let $\tilde \delta =\min (c_0, \gamma)$, $\tilde \kappa = c_1 \kappa$  and thus the conclusion follows.
\end{proof}

\subsection{Rescaled situation}

Fix some $x_{0}\in \mathbb{R}^n$ with $|x_{0}|\geq M$ where $M$ is large enough and will be adjusted later. Let $R\rho=2|x_0|$ where $\rho$  will be chosen as in Theorem \ref{estimates-for-resc} corresponding to $\tilde \kappa$  in Lemma \ref{com-conse1}.  Then given $\tilde t \in [t_0, t_0 +\tilde \delta)$ with $\tilde \delta$ as in Lemma \ref{com-conse1}, the rescaled function \begin{equation}\label{u_r} U_{R} (X,t):=U(RX+(x_{0}, 0), R^2 t+\tilde t)\end{equation} satisfies the following estimate  as a consequence of Lemma~\ref{com-conse1}
\begin{align}\label{lob}
R^{(n+a+1)}\int_{\mathbb{B}_{\rho}^{+}} U_{R}^2(X, 0) x_{n+1}^{a}\, dX=\int_{\mathbb{B}_{2|x_0|}^{+}((x_{0}, 0))} U^2(X, \tilde t) x_{n+1}^{a}\, dX\geq\int_{\mathbb B_{1}^{+}} x_{n+1}^{a}U^2(X, \tilde t) dX\geq \tilde{\kappa}.
\end{align}
Here onwards we shall look into the rescaled scenario and derive results for the rescaled function $U_R$ and eventually we will scale back to $U$. We have that corresponding to $U$ in \eqref{exprob},  $U_{R}$ satisfies the following equation:
\begin{equation}\label{exprob-un-rep}
		\begin{cases}
			x_{n+1}^a \partial_t U_{R} + \operatorname{div}(x_{n+1}^a \nabla U_{R})=0\ \ \ \ \ \ \ \ \ \ \ \ \text{in} \ \mb{B}_5^+ \times [0, \frac{1}{R^2}),
			\\	
			U_{R}(x,0, t)= u_{R}(x,t)
			\\
			\py U_R(x, 0,t)= R^{2s}V_{R}U_R\ \ \ \ \ \ \ \ \ \text{in}\ B_5 \times [0, \frac{1}{R^2}),
		\end{cases}
	\end{equation}
\text{where} \begin{equation}\label{urvr}  V_{R}(x, t):=V(Rx+(x_{0},0), R^2t+\tilde t).	\end{equation}
We  now  derive our first monotonicity result which is the nonlocal counterpart of \cite[Lemma 1]{EKVP}. It is to be mentioned that although similar results  have appeared in the  previous works \cite{ABDG, AA} which deals with the local asymptotic of solutions to \eqref{exprob},  the new feature of the result in Lemma \ref{lemma-mont} below is the validity of a similar monotonicity result in time   for $t \in [0, 1/R^2]$ under a certain asymptotic behaviour (in $R$) of the weighted Dirichlet to Neumann map  as $R \to \infty$. More precisely, we are interested in deriving an inequality as in \eqref{mon-un-1} below when the zero order perturbation $\tilde V  := R^{2s} V_R$ of the weighted Neumann derivative  $\py U_R$ satisfies  $|| \tilde V||_{L^\infty} \leq R^{2s}$. Note that such a bound on $\tilde V$ holds   in view of \eqref{vasump}.

\begin{lemma}[Monotonicity]
\label{lemma-mont}
Let $U_R$ be as in  \eqref{u_r} and \begin{equation}\label{use1}R^{(n+a+1)}\int_{\mathbb{B}_{\rho}^{+}} U_R^2(X, 0) x_{n+1}^{a}\, dX\geq \tilde{\kappa}, \end{equation} for some $\tilde{\kappa}, \rho\in (0, 1)$ and $R\geq 10.$ Then there exists a large universal constant $M=M(n, a, \kappa)$ such that 
\begin{align}
M\int_{\mathbb{B}_{2\rho}^{+}} U_R^2(X, t) x_{n+1}^{a}\, dX\geq R^{-(n+a+1)},\ \text{(which follows from \eqref{lob})},
\label{mon-un-1}
\end{align}
for all $0\leq t\leq \f{\ c}{R^2},$ where $ c$  is  sufficiently small.
\end{lemma}
\begin{proof}
For simplicity, we show it for $\rho=1$. Let  $f= \phi\, U_R,$ where $\phi \in C_0^{\infty}(\mathbb B_2)$ is a spherically symmetric cutoff such that $0\le \phi\le 1$ and $\phi \equiv1$ on $\mathbb B_{3/2}.$ Considering the symmetry of $\phi$ in $x_{n+1}$ variable and the fact that $U_R$ solves \eqref{exprob-un-rep}, we obtain
		\begin{equation}\label{feq-un}
			\begin{cases}
				x_{n+1}^a f_t + \operatorname{div}(x_{n+1}^a \nabla f) = 2 x_{n+1}^a \langle\nabla U,\nabla \phi\rangle  + \operatorname{div}(x_{n+1}^a \nabla \phi) U\ \ \ \ \ \ \ \text{in} \ \mb{B}_5^+ \times [0, \frac{1}{R^2}),
				\\	
				f(x,0,t)= u(x,t)\phi(x,0)
				\\
				\py f(x,0, t)= R^{2s} V_{R}f\ \ \ \ \ \ \ \ \ \ \ \ \ \ \ \ \ \ \ \ \ \ \ \ \ \ \  \text{in}\ {B}_5 \times [0, \frac{1}{R^2}).
			\end{cases}
		\end{equation}
Define
\begin{align*}
			H(t) = \int_{\mb{R}^{n+1}_+} x_{n+1}^a f(X,t)^2 \mc{G}(Y,X,t) dX,
		\end{align*}
		where $\mc{G}(Y,X,t) = p(y, x, t) \ p_a(x_{n+1},y_{n+1};t),$ and  $p(y,x,t)$ is the heat-kernel associated to $\left(\dd_t - \Delta_{x}\right)$ and $p_a$ is the fundamental solution of  the Bessel operator $\dd^2_{x_{n+1}} + \f{a}{x_{n+1}}\dd_{x_{n+1}}$. It is well-known that $p_a$ is given by the formula
   \begin{equation}\label{pa}
   	p_a(x_{n+1}, y_{n+1}; t) = (2t)^{- \f{1+a}{2}} e^{-\f{x_{n+1}^2 + y_{n+1}^2}{4t}} \left( \f{x_{n+1} y_{n+1}}{2t} \right)^{ \f{1-a}{2}} I_{\f{a-1}{2}}\left( \f{x_{n+1} y_{n+1}}{2t} \right),
   \end{equation}
   where $I_\nu (z)$ the modified Bessel function of the first kind defined by the series
   
   \begin{align}\label{besseries}
   	I_{\nu}(z) = \sum_{k=0}^{\infty}\frac{(z/2)^{\nu+2k} }{\Gamma (k+1) \Gamma(k+1+\nu)}, \hspace{4mm} |z| < \infty,\; |\operatorname{arg} z| < \pi.
   \end{align}
   Also, for $t>0$, $\mc{G} = \mc{G}(Y, \cdot)$ solves $\operatorname{div}(x_{n+1}^a \nabla \mc{G}) = x_{n+1}^a \partial_t \mc{G}.$  We refer the reader to \cite{Gcm} for the relevant details. Differentiating with respect to $t$, we find 

\begin{align}\label{hprime}
			H'(t) & = 2 \int x_{n+1}^a f f_t \mc{G} + \int x_{n+1}^a f^2 \dd_t\mc{G} \\
			& = 2 \int x_{n+1}^a f f_t \mc{G} + \int f^2 \tn{div}\left(x_{n+1}^a \nabla \mc{G} \right)\notag \\
   & = 2 \int x_{n+1}^a f f_t \mc{G} - \int x_{n+1}^a \langle \nabla(f^2),  \nabla \mc{G} \rangle \notag \\
   & = 2 \int x_{n+1}^a f f_t \mc{G}+ \int \operatorname{div}( x_{n+1}^a \nabla(f^2))\mc{G}+2 R^{2s}\int_{\{x_{n+1}=0\}} V_Rf^2 G\notag \\
			& = 2 \int f \mc{G} \left( x_{n+1}^a f_t + \tn{div} \left(x_{n+1}^a \cdot \nabla f \right) \right) + 2\int x_{n+1}^a  \mc{G} |\nabla f|^2+2 R^{2s}\int_{\{x_{n+1}=0\}} V_Rf^2 G\notag\\
   &=J_1+J_2+J_3.\label{ja}
		\end{align}

\begin{itemize}
\item For every $Y\in \mathbb B_1^+$ and $0<t\le \f{1}{R^2}$ we have (keeping in mind equation (3.13) in \cite{ABDG})
\begin{equation}\label{i1}
J_1 \geq - C e^{-\frac{1}{N t}} N R^{4}.\end{equation}

This can be seen as follows. Following the proof of the inequality (3.13) in \cite{ABDG}, we find

\begin{equation}\label{j1}
|J_1| \leq C e^{-\frac{1}{Nt}} \int_{\mb{B}_2^+} x_{n+1}^a (|\nabla U_R|^2 + U_R^2).\end{equation}
Since $U_R$ solves \eqref{exprob-un-rep}, by invoking the $L^{\infty}$ bounds on $U_R, x_{n+1}^a \partial_{x_{n+1}}U_R, \nabla_x U_R$ using Lemma \ref{reg1}, we find that \eqref{i1} follows. We then observe that since $t \leq 1/R^2$, for a different $N$, it follows from \eqref{i1} that the following holds
\begin{equation}\label{i10}
J_1 \geq Ce^{-\frac{1}{Nt}}.
\end{equation}

\item 
We now recall the  inequality in \cite[(3.21)]{ABDG}.   Keeping in mind that only $L^\infty$ norm of $R^{2s}V_{R}$ appears in the expression, we find that  for every $Y\in \mathbb B_1^+$ and $0<t\le 1/R^2$ one has
\begin{align}
|J_3| & \notag \leq C(n,a)R^{2s}\bigg(A^{1+a}\int f^2 \mc G x_{n+1}^a dX + \frac{n+a+1}{4t} A^{a-1} \int f^2 \mc G  x_{n+1}^a dX 
\\
\notag& + A^{a-1} \int |\nabla f|^2 \mc G  x_{n+1}^a dX\bigg)\\
& \notag \leq C(n,a) R^{2s}\bigg(t^{-\f{1+a}{2}}\int f^2 \mc G x_{n+1}^a dX + t^{-\f{1+a}{2}} \int f^2 \mc G  x_{n+1}^a dX 
\\
& + t^{\f{1-a}{2}} \int |\nabla f|^2 \mc G  x_{n+1}^a dX\bigg)\,\,\,(\text{putting}\,\,A\sim \f{1}{\sqrt{t}}).
\label{mk1}
\end{align}
\end{itemize}
Combining \eqref{i10} and \eqref{mk1} we obtain
\begin{align}
\notag H'(t)&\geq - C e^{-\frac{1}{N t}} +2 \int x_{n+1}^a  \mc{G} |\nabla f|^2\\
&-C R^{2s}t^{-\f{1+a}{2}} H(t) -C R^{2s}t^{\f{1-a}{2}} \int |\nabla f|^2 \mc G  x_{n+1}^a dX.\label{h1-un}
\end{align}
For $0\leq t\leq \f{c}{R^2}$  using \eqref{sa} we have $R^{2s}t^s\ll 1,$ provided $c$ is sufficiently small.  This in turn ensures that  the second term absorbs the last one in \eqref{h1-un}. Thus we find
\begin{align}
H'(t)&\geq - C e^{-\frac{1}{N t}} -C R^{2s}t^{-\f{1+a}{2}} H(t).
\label{h2-un}
\end{align}
As a conclusion we get
\begin{align}
\left(e^{ CR^{2s}t^{\f{1-a}{2}}} H(t)\right)'\geq -C e^{R^{2s}t^{\f{1-a}{2}}} e^{-\frac{1}{N t}} \label{ekta-imp}. 
\end{align}
Keeping in mind that $0<t\leq \f{c}{R^2},$ integrating \eqref{ekta-imp} from $0$ to $t$ we get using \begin{equation}\label{lim}
		\lim_{t \to 0^+} H(t) = U_R (Y, 0)^2\ \text{(see \cite[(3.6]{ABDG})}, \end{equation}
that the following inequality holds
\begin{align*}
& e^{CR^{2s}t^{\f{1-a}{2}}} H(t)- U_R(Y, 0)^2\geq -C N\int_{0}^{t} e^{R^{2s}\eta^{\f{1-a}{2}}} e^{-\frac{1}{N \eta}} d\eta\\
&\implies M H(t)\geq U_R(Y, 0)^2-C N  t e^{R^{2s}t^{\f{1-a}{2}}} e^{-\frac{1}{N t}}.
\end{align*}

Again integrating  with respect to  $Y$  in $\mathbb B_1^+$ and exchanging the order of integration,  using $\int \mc{G}(Y,X,t) y_{n+1}^a dY=1$ and by   renaming the variable $Y$ as $X$  we obtain using \eqref{use1}
\begin{align*}
M \int_{\mathbb B_2^+}U_R(X,t)^2 x_{n+1}^a dX&\geq \int_{\mathbb B_1^+}U_R(X,0)^2 x_{n+1}^a dX -C N  t e^{R^{2s}t^{\f{1-a}{2}}} e^{-\frac{1}{N t}}\\
&\geq \tilde{\kappa} R^{-(n+a+1)}-C N t e^{R^{2s}t^{\f{1-a}{2}}} e^{-\frac{1}{N t}}\gtrsim R^{-(n+a+1)},
\end{align*}
where we have used that for $0\leq t\leq \f{ c}{ R^2},$ $e^{R^{2s}t^{\f{1-a}{2}}}$ is uniformly bounded and the quantity $e^{-\frac{1}{N t}}$ can be made suitably small.   The conclusion thus follows.
\end{proof}

We  now state and prove our  main  Carleman estimate in the rescaled setting \eqref{exprob-un-rep}  which is needed to obtain the desired lower bounds at infinity for  solutions to the extension problem \eqref{exprob}. As remarked earlier, the main new feature  of Theorem \ref{carl1-un} is the validity   of the Carleman estimate   in \eqref{care1-un} below  in presence of   the prescribed  limiting behaviour (in $R$) of the weighted Dirichlet to Neumann map as $R \to \infty$.
\begin{theorem}[Main Carleman estimate]
\label{carl1-un}
Let $s\in [\f{1}{2}, 1)$ and $ \widetilde{\mc{H}}_s$ be the   backward in time extension operator  in \eqref{extop}. Let $ w \in C_0^\iy \left( \overline{\mb{B}^+_4} \times [0,\left. \f{1}{e\ld}\right) \right)$ where $ \ld = \frac{\alpha}{ \delta^2}$ for some  $ \delta \in (0,1)$ sufficiently small. Furthermore, assume that  $\py w\equiv R^{2s}V_{R}w$ on $\{x_{n+1}=0\}$(with $V_R$ as in \eqref{urvr}) and
\begin{equation}\label{al}
\alpha\geq M R^{2},
\end{equation} where $M$ is a large universal constant. Then  the following estimate holds
		\begin{align}\label{care1-un}
			& \alpha^2 \int_{\mb{R}^{n+1}_+ \times [c, \iy)} x_{n+1}^{a} \sigma_{s}^{-2 \alpha}(t) \ w^2 \ G  + \alpha \int_{\mb{R}^{n+1}_+ \times [c, \iy)} x_{n+1}^{a} \sigma_{s}^{1-2 \alpha}(t)\  |\nabla w|^2 \ G  \\
			& \leq M  \int_{\mb{R}^{n+1}_+ \times [c, \iy)} \sigma_{s}^{ 1-2 \alpha}(t) x_{n+1}^{-a} \ \lvert \widetilde{\mc{H}_{s}} w \rvert ^2 \ G\notag\\
   &+ \sigma_{s}^{-2 \alpha}(c) \left\{ -\f{c}{M} \int_{t=c} x_{n+1}^a \ |\nabla w(X,c)|^2 \ G(X,c) \ dX + M\alpha \int_{t=c} x_{n+1}^a \ |w(X,c)|^2 \ G(X,c) \ dX\right\}.\notag
		\end{align}
		Here $\sigma_{s}$ is as in Lemma \ref{sigma},  $G(X,t) = \f{1}{{t^\f{n+1+a}{2}}} e^{-\f{|X|^2}{4t}}$ and $0 < c \le \f{1}{5\ld}. $
	\end{theorem}
	\begin{proof}
We  partly follow the arguments as   in  the proof of Theorem 3.1 in \cite{AA}. However  the reader will notice that  the proof of the estimate \eqref{care1-un} involves  some very delicate adaptations due to the presence of  an ``amplified'' boundary condition as in \eqref{exprob-un-rep} for  $R \to \infty$.  Before proceeding further, we mention that  throughout the proof, the solid integrals below will be taken in $ \mb{R}^n \times [c, \iy) $ where $ 0 < c \le \f{1}{\ld} $ and we refrain from mentioning explicit limits in the rest of our discussion. Note that 
		\[ x_{n+1}^{-\f{a}{2}} \widetilde{\mc{H}}_s = x_{n+1}^{\f{a}{2}} \left( \dd_t + \tn{div}(\nabla) + \f{a}{x_{n + 1}} \dd_{n+1}\right).\]
		Define 
		\[ w(X,t) = \sigma_{s}^\alpha(t) e^{\f{|X|^2}{8t}} v(X,t). \]

Therefore,
\begin{align*}
\tn{div} (\nabla w)& = \tn{div} \left( \sigma_{s}^\alpha(t) e^{\f{|X|^2}{8t}}  \left( \nabla v + \f{X}{4t} v \right) \right) \\ 
& = \sigma_{s}^\alpha(t) e^{\f{|X|^2}{8t}} \left[ \tn{div} (\nabla v) + \f{\langle X,  \nabla  v \rangle} {2t} + \left( \f{|X|^2}{16 t^2} + \f{n+1} {4t} \right) v \right].
\end{align*}
		Now we define the vector field 
		\begin{equation}\label{defz} \mc{Z} := 2t \dd_t + X \cdot \nabla. \end{equation}
		
		Note that $\mc{Z}$ is the infinitesimal generator of the parabolic dilations $\{\delta_r\}$ defined by $\delta_r(X,t)=(rX, r^2 t)$. Then
\begin{align*}
			x_{n+1}^{-\f{a}{2}} \sigma_{s}^{-\alpha}(t) e^{-\f{|X|^2}{8t}} \widetilde{\mc{H}}_sw 
			& = x_{n+1}^{\f{a}{2}}  \left[ \tn{div}\left( \nabla v\right) + \f{1}{2t} \mc{Z}v + \left( \f{n+1 + a} {4t} + \f{\alpha \sigma_{s}'} {\sigma_{s}}\right) v - \f{|X|^2}{16t^2} v  + \f{a}{x_{n+1}} \dd_{n+1} v \right].
		\end{align*}

		Next we consider the expression
		\begin{align}\label{ex1}
			&  \int \sigma_{s}^{-2 \alpha}(t) t^{-\mu} x_{n+1}^{-a} e^{-\f{|X|^2}{4t}} \left(\f{t \sigma_{s}'}{\sigma_{s}}\right)^{-\f{1}{2}}  \lvert \widetilde{\mc{H}_s} w \rvert ^2\notag \\ 
			&= \int x_{n+1}^{a} t^{-\mu} \left(\f{t \sigma_{s}'}{\sigma_{s}}\right)^{-\f{1}{2}} \left[ \tn{div}\left( \nabla v\right) + \f{1}{2t} \mc{Z}v + \left( \f{n+1 + a} {4t} + \f{\alpha \sigma_{s}'} {\sigma_{s}}\right) v - \f{|X|^2}{16t^2} v  + \f{a}{x_{n+1}} \dd_{n+1} v \right]^2,
		\end{align}
		where \begin{equation}\label{mu} \mu= \frac{n-1+a}{2}.\end{equation} Then we estimate the integral \eqref{ex1} from below with an application of the algebraic inequality
		\[ \int P^2 + 2 \int PQ \le \int \left( P + Q \right)^2,\]
		where $P$ and $Q$ are chosen as
		\begin{align*}
			& P = \f{x_{n+1}^{\f{a}{2}} t^{-\f{\mu + 2}{2}}}{ 2} \left(\f{t \sigma_{s}'}{\sigma_{s}}\right)^{-\f{1}{4}} \mc{Z}v, \\ 
			& Q = x_{n+1}^{\f{a}{2}} t^{-\f{\mu}{2}} \left(\f{t \sigma_{s}'}{\sigma_{s}}\right)^{-\f{1}{4}}\left[ \tn{div}\left( \nabla v\right)  + \left( \f{n+1 + a} {4t} + \f{\alpha \sigma_{s}'} {\sigma_{s}}\right) v - \f{|X|^2}{16t^2} v  + \f{a}{x_{n+1}} \dd_{n+1} v \right].
		\end{align*} 
		We compute the terms coming from the cross product, i.e. from $\int PQ.$ We write $$ \int PQ: = \sum_{k=1}^4 \mc{I}_k,$$ where
		\begin{align*}
			&  \mc{I}_1 = \int x_{n+1}^{a} t^{-\mu} \left(\f{t \sigma_{s}'}{\sigma_{s}}\right)^{-\f{1}{2}} \f{1}{2t} \mc{Z}v \left( \f{n+1 + a} {4t} + \f{\alpha \sigma_{s}'} {\sigma_{s}} \right)v, \\
			&  \mc{I}_2 = \int x_{n+1}^a t^{-\mu} \left(\f{t \sigma_{s}'}{\sigma_{s}}\right)^{-\f{1}{2}} \f{\mc{Z}v} {2t} \ \tn{div} \left( \nabla v \right), \\
			& \mc{I}_3 = \int x_{n+1}^a t^{-\mu} \left(\f{t \sigma_{s}'}{\sigma_{s}}\right)^{-\f{1}{2}} \f{\mc{Z}v} {2t} \left(- \f{|X|^2}{16 t^2}\right) v, \\
			&  \mc{I}_4 = \int x_{n+1}^{a} t^{-\mu} \left(\f{t \sigma_{s}'}{\sigma_{s}}\right)^{-\f{1}{2}} \f{\mc{Z}v} {2t} \f{a\, \dd_{n+1}v}{x_{n+1}}.
		\end{align*}
		
Following the arguments in \cite{AA}, we obtain ( see \cite[(3.22)]{AA})
\begin{align}
\notag &\mc{I}_{1}+\mc{I}_{2}+ \mc{I}_{3}+ \mc{I}_{4} \\
&\geq \f{\alpha}{N} \int x_{n+1}^{a} \sigma_{s}^{-2\alpha}(t)  \f{\theta_{s}(\lambda t)}{t}\, G\, w^2+ \frac{1}{N}  \int x_{n+1}^a \f{\theta_{s}(\lambda t)}{t} \sigma_{s}^{1-2\alpha}(t)\,\, G\,\, |\nabla w|^2\notag\\
&-N\alpha \sigma_{s}^{-2\alpha}(c) \int_{t=c} x_{n+1}^{a} w^2(X,c)\,G(X, c)+\f{c}{N}\, \sigma_{s}^{-2\alpha}(c) \int_{t=c} x_{n+1}^{a} |\nabla w|^2 G \ dX\notag\\
&-\f{1}{2} R^{2s}\int_{\{x_{n+1}=0\}} t^{-\mu-1}\left(\f{t \sigma_{s}'}{\sigma_{s}}\right)^{-\f{1}{2}} {V}_{R}(x, t)\, v\, \langle x, \nabla_{x}v\rangle-R^{2s}\int_{\{x_{n+1}=0\}}t^{-\mu}\left(\f{t\sigma_{s}'}{\sigma_{s}}\right)^{-1/2}  V_{R}(x, t) \partial_{t}\left(\f{v^2}{2}\right).\label{combined-est}
\end{align}
Let us estimate the boundary terms in \eqref{combined-est}. Using the divergence theorem we obtain the following alternate representation of such boundary terms.  

 \begin{align*}
 &K_{1}:=\f{1}{4} R^{2s} \int_{\{x_{n+1}=0\}} t^{-\mu-1}\left(\f{t \sigma_{s}'}{\sigma_{s}}\right)^{-\f{1}{2}} (V_{R}(x, t) n+\langle x, \nabla_{x} V_{R}(x, t)\rangle ) v^2,\\
 &K_{2}:=-R^{2s}\int_{\{x_{n+1}=0\}}t^{-\mu}\left(\f{t\sigma_{s}'}{\sigma_{s}}\right)^{-1/2}  V_{R}(x, t) \partial_{t}\left(\f{v^2}{2}\right).
 \end{align*}
 It is to be noted that using \eqref{vasump} and \eqref{urvr} we have
 \begin{equation}\label{ubd}
 |\nabla_x V_R| \leq R, \ |\partial_t V_R| \leq R^2.\end{equation}
 Using the trace inequality Lemma \ref{tr} and \eqref{ubd} we find 
\begin{align}
&|K_{1}|
 \lesssim  R^{2s+1} \int t^{-\mu-1} \sigma_{s}^{-2\alpha} \int_{\Rb^n} e^{-\frac{|x|^2}{4t}} w^2 \label{k1mid}\\
\notag & \lesssim R^{2s+1} \int t^{-\mu-1} \sigma_{s}^{-2\alpha} \left(A(t)^{1+a}\int_{\Rb^{n+1}_{+}} x_{n+1}^a e^{-\frac{|X|^2}{4t}} w^2+A(t)^{a-1}\int_{\Rb^{n+1}_{+}} x_{n+1}^a\big|\nabla w-w \f{X}{4t}\big|^2 e^{-\frac{|X|^2}{4t}} \right)\\
\notag & \lesssim R^{2s+1}\int t^{-\mu-1} \sigma_{s}^{-2\alpha} \left(A(t)^{1+a}\int_{\Rb^{n+1}_{+}} x_{n+1}^a e^{-\frac{|X|^2}{4t}} w^2+A(t)^{a-1}\int_{\Rb^{n+1}_{+}} x_{n+1}^a |\nabla w|^2 e^{-\frac{|X|^2}{4t}}\right.\\
&\left.+A(t)^{a-1}\int_{\Rb^{n+1}_{+}} x_{n+1}^{a} w^2 \f{|X|^2}{16t^2} e^{-\frac{|X|^2}{4t}} \right)\label{trace}
\end{align}
for $A(t)>1$. The choice of $A(t)$ will be crucial to complete our proof. Also, it follows from the Hardy inequality in Lemma \ref{hardy} that the following estimate holds
\begin{align}
\label{Hardy}
\int_{\Rb^{n+1}_{+}} x_{n+1}^{a} w^2 \f{|X|^2}{16t^2} e^{-\frac{|X|^2}{4t}}\leq \int_{\Rb^{n+1}_{+}} x_{n+1}^a \f{n+1+a}{4t} e^{-\frac{|X|^2}{4t}} w^2+\int_{\Rb^{n+1}_{+}} x_{n+1}^a e^{-\frac{|X|^2}{4t}} |\nabla w|^2.
\end{align}
Plugging the estimate \eqref{Hardy} in \eqref{trace}  and by using \eqref{mu} yields
\begin{align}
\notag |K_1|& \lesssim R^{2s+1} \left( \int A(t)^{1+a}  \sigma_{s}^{-2\alpha}\,  x_{n+1}^a G w^2+2\int A(t)^{a-1}x_{n+1}^a \sigma_{s}^{-2\alpha} |\nabla w|^2 G\right.\\
&\left.+\int A(t)^{a-1}\sigma_{s}^{-2\alpha-1} x_{n+1}^a  G\, w^2 \right).\label{k1est}
\end{align}
In the last inequality in \eqref{k1est} above, we used that $\sigma_s(t) \sim t$. 
Now we choose $A(t)>1$ in such a way that the above terms can be absorbed in the positive terms on the right hand side in \eqref{combined-est} above, i.e. in the terms  $\frac{\alpha}{N} \int x_{n+1}^{a} \sigma_{s}^{-2\alpha}(t) \f{\theta_{s}(\ld t)}{t} \ w^2  G$ and $\frac{1}{N} \int x_{n+1}^a \frac{\theta_{s}(\lambda t)}{t}  \sigma_{s}^{1-2\alpha}(t) |\nabla w|^2 G$. Therefore given the value of  $\mu$ as in \eqref{mu},  we require
\begin{equation}\begin{cases} A(t)^{1+a}R^{2s+1}\lesssim  \f{\alpha}{10N} \frac{\theta_{s}(\lambda t)}{t},\\ A(t)^{a-1} R^{2s+1}\lesssim \f{1}{10N}  \theta_{s}(\lambda t), \\ \frac{A(t)^{a-1}}{t} R^{2s+1}  \lesssim  \f{\alpha}{10N} \frac{\theta_{s}(\lambda t)}{t}.\end{cases}
\label{cases}
\end{equation}
It is easy to see that the third inequality automatically holds if the second one is satisfied since $\alpha$ is to be chosen large. Therefore, it is sufficient to choose $A(t)$ satisfying the first two inequalities. Recall that $a=1-2s,$ and if we set 
$$A(t)=\left(\f{10N  R^{2s+1}}{\theta_{s}(\lambda t)}\right)^{1/2s},$$ then the second inequality in \eqref{cases} is valid. Note that $A(t)>1$ as $\theta_{s}(t)\to 0$ as $t\to 0.$ Moreover, the above choice of $A$ will also satisfy the first inequality in \eqref{cases} if 
\begin{align}
\notag &\left(\f{10N R^{2s+1}} {\theta_{s}(\lambda t)}\right)^{\f{2(1-s)}{2s}} R^{2s+1}\leq \f{\alpha}{10N} \frac{\theta_{s}(\lambda t)}{t}\\
\notag\iff& \left(\f{10N  R^{2s+1}}{\theta_{s}(\lambda t)}\right)^{\f{1}{s}}\f{\theta_{s}(\lambda t)}
{10N R^{2s+1}} R^{2s+1}\leq \f{\alpha}{10N} \frac{\theta_{s}(\lambda t)}{t}\\
\label{qu}\iff & 10N  R^{2s+1}\leq \alpha^s t^{-s} \theta_{s}(\lambda t).
\end{align}
Finally, observe that $\theta_{s}(\lambda t)=(\lambda t)^s\left(\log \f{1}{\lambda t} \right)^{1+s}\geq (\lambda t)^s$ since $\log \f{1}{\lambda t}\geq 1$ on $[0, \f{1}{e\lambda}],$ so the inequality \eqref{qu} is ensured if we choose  $\alpha$ large enough such that $$ \alpha^s t^{-s} (\lambda t)^s\geq 10N  R^{2s+1}.$$
Consequently since $\lambda=\alpha \delta^2,$ by choosing some arbitrary $\delta\in (0, 1),$ we conclude that the choice of $A(t)$ above satisfy the set of inequalities in \eqref{cases} provided $$\alpha^{2s}\geq (1+ N)R^{2s+1}.$$
The above is ensured for  $\alpha\geq MR^2$ with $M$  large and $R>1$ provided $s\in [\frac{1}{2}, 1).$

For $K_2,$ applying integration by parts we observe
\begin{align}
|K_2|&=\bigg|R^{2s}\f{1}{2}\int_{\{x_{n+1}=0\}}(-\mu)t^{-\mu-1}\left(\f{t\sigma_{s}'}{\sigma_{s}}\right)^{-1/2}V_{R}(x, t) v^2+R^{2s}\f{1}{2}\int_{\{x_{n+1}=0\}} t^{-\mu}\f{-1}{2}\left(\f{t\sigma_{s}'}{\sigma_{s}}\right)^{-3/2} \left(\f{t\sigma_{s}'}{\sigma_{s}}\right)'  V_{R} v^2 \notag\\
&+R^{2s}\f{1}{2}\int_{\{x_{n+1}=0\}} t^{-\mu}\left(\f{t\sigma_{s}'}{\sigma_{s}}\right)^{-1/2} \partial_{t}(V_{R}) v^2+R^{2s}\f{1}{2}\int_{\{x_{n+1}=0; t=c\}} c^{-\mu}\left(\f{c\sigma_{s}'(c)}{\sigma_{s}(c)}\right)^{-1/2} V_{R}(x, c) v^2(x, c)\notag \bigg|.
\end{align}
Using \eqref{ubd},  the fact that  $\left(\f{t\sigma_{s}'}{\sigma_{s}}\right)\sim 1$ and also that  $0\leq t<\f{1}{R^{2}}, $ we observe that the first and third terms on the right hand side of the above expression can be  bounded by $$C  R^{2s}\int_{\{x_{n+1}=0\}} t^{-\mu-1} \sigma_{s}^{-2\alpha} e^{-|x|^2/4t} w^2.$$   
The second term is dominated by $R^{2s} \int_{\{x_{n+1}=0\}} t^{-\mu} \bigg|-\left(\f{t\sigma_{s}'}{\sigma_{s}}\right)'\bigg| v^2,$ which in turn is bounded by  
$$C R^{2s} \int_{\{x_{n+1}=0\}} t^{-\mu-1} \sigma_{s}^{-2\alpha} e^{-|x|^2/4t} w^2,$$
considering the fact that $-\left(\f{t\sigma_{s}'}{\sigma_{s}}\right)'$ is comparable to $\f{\theta_{s}(\lambda t)}{t}$ and $\theta_{s}(\lambda t)\to 0$ as $t\to 0.$ Combining the above arguments we have
\begin{align}
|K_2|&\lesssim R^{2s} \int_{\{x_{n+1}=0\}} t^{-\mu-1} \sigma_{s}^{-2\alpha} e^{-|x|^2/4t} w^2+\bigg|R^{2s}\f{1}{2}\int_{\{x_{n+1}=0\}} c^{-\mu}\left(\f{c\sigma_{s}'(c)}{\sigma_{s}(c)}\right)^{-1/2} V_{R}(x, c) v^2(x, c) \bigg|.\label{bdry} 
\end{align}
The first term in \eqref{bdry} can be handled similarly as $K_1,$ see \eqref{k1mid}-\eqref{qu}. For the last term in \eqref{bdry}, using trace inequality and performing similar calculations as in \eqref{k1est}, we obtain that
\begin{align}
\notag
 &\bigg|R^{2s} \f{1}{2}\int_{\{x_{n+1}=0; t=c\}} c^{-\mu}\left(\f{c\sigma_{s}'(c)}{\sigma_{s}(c)}\right)^{-1/2} V_{R}(x, c) v^2(x, c) \bigg|\\
\notag &\lesssim c R^{2s} \int_{\{x_{n+1}=0\}} c^{-\mu-1} \sigma_{s}^{-2\alpha}(c)\,e^{-\f{|x|^2}{4c}} w^2(x, c)\\
&\lesssim \notag  R^{2s} \left(c \sigma_{s}^{-2\alpha}(c) A^{1+a}\int  \,  x_{n+1}^a G(X, c) w^2(X, c)+2c A^{a-1}\sigma_{s}^{-2\alpha}(c)\int x_{n+1}^a  |\nabla w(X, c)|^2 G(X, c)\right.\\
&\left.+A^{a-1}\,c \f{n+1+a}{4c}\int \sigma_{s}^{-2\alpha}(c) x_{n+1}^a  G(X, c)\, w^2(X, c) \right)\label{name}
\end{align}
holds for any $A>1.$ If we now choose $A$ sufficiently large, say \begin{equation}\label{newa}A^{2s}\sim 100N R^{2s},\end{equation}  then the term $$2c R^{2s} A^{a-1}\sigma_{s}^{-2\alpha}(c)\int x_{n+1}^a  |\nabla w(X, c)|^2 G(X, c)$$ in \eqref{name} can easily be absorbed by the term $\f{c}{N} \, \sigma_{s}^{-2\alpha}(c) \int_{t=c} x_{n+1}^{a} |\nabla w|^2 G \ dX$ in \eqref{combined-est}. Corresponding to this choice of $A$ as in \eqref{newa}, we find by also  using that $c \lesssim \frac{1}{\alpha}\sim \frac{1}{R^2}$, the remaining terms in the last expression in \eqref{name} above can be estimated as
\begin{align*}
&R^{2s} \left( c \sigma_{s}^{-2\alpha}(c) A^{1+a}\int  \,  x_{n+1}^a G(X, c) w^2(X, c)+A^{a-1}\,c \f{n+1+a}{4c}\int \sigma_{s}^{-2\alpha}(c) x_{n+1}^a  G(X, c)\, w^2(X, c) \right) 
\\
&\leq N\alpha \sigma_{s}^{-2\alpha}(c) \int_{t=c} x_{n+1}^{a} w^2(X,c)\,G(X, c).\end{align*}

 Therefore, from the above discussion, the contributions from $K_1$ and $K_2$ can be absorbed appropriately by the first four terms in \eqref{combined-est} so that for large $\alpha$ satisfying $\alpha\geq MR^2$ for a large $M$ the following holds
\begin{align}
\label{finalest}
&\int \sigma_{s}^{-2 \alpha}(t) t^{-\mu} x_{n+1}^{-a} e^{-\f{|X|^2}{4t}} \left(\f{t \sigma_{s}'}{\sigma_{s}}\right)^{-\f{1}{2}}  \lvert \widetilde{\mc{H}_s} w \rvert ^2\\
&\geq \mc{I}_{1}+\mc{I}_{2}+ \mc{I}_{3}+ \mc{I}_{4}\notag\\
&\geq\f{\alpha}{N} \int x_{n+1}^{a} \sigma_{s}^{-2\alpha}(t)  \f{\theta_{s}(\lambda t)}{t}\, G\, w^2+ \frac{1}{N}  \int x_{n+1}^a \f{\theta_{s}(\lambda t)}{t} \sigma_{s}^{1-2\alpha}(t)\,\, G\,\, |\nabla w|^2\notag\\
&-N\alpha \sigma_{s}^{-2\alpha}(c) \int_{t=c} x_{n+1}^{a} w^2(X,c)\,G(X, c)+\f{c}{N}\, \sigma_{s}^{-2\alpha}(c) \int_{t=c} x_{n+1}^{a} |\nabla w|^2 G \ dX.\notag
\end{align}
Also, we have $\f{\theta_{s}(\lambda t)}{t}\gtrsim \lambda=\f{\alpha}{\delta^2},$ hence 
\begin{align}
\notag& N\int \sigma_{s}^{-2 \alpha}(t) t^{-\mu} x_{n+1}^{-a} e^{-\f{|X|^2}{4t}} \left(\f{t \sigma_{s}'}{\sigma_{s}}\right)^{-\f{1}{2}}  \lvert \widetilde{\mc{H}_s} w \rvert ^2\\
&\geq \alpha^2 \int x_{n+1}^{a} \sigma_{s}^{-2\alpha}(t) \ w^2  G +\alpha \int x_{n+1}^a   \sigma_{s}^{1-2\alpha}(t) |\nabla w|^2 G\notag\\
 &-N\alpha \sigma_{s}^{-2\alpha}(c) \int_{t=c} x_{n+1}^{a} w^2(X,c)\,G(X, c)+\f{c}{N}\, \sigma_{s}^{-2\alpha}(c) \int_{t=c} x_{n+1}^{a} |\nabla w|^2 G \ dX\label{finalest-1}
\end{align}
possibly for a new universal constant $N.$ Finally, the  conclusion follows from \eqref{finalest-1} since $$\int_{\mb{R}^{n+1}_+ \times [c, \iy)}  \sigma_{s}^{-2 \alpha}(t) t^{-\mu} x_{n+1}^{-a} e^{-\f{|X|^2}{4t}} \left(\f{t \sigma'}{\sigma}\right)^{-\f{1}{2}}  \lvert \widetilde{\mc{H}_s} w \rvert ^2 \sim \int_{\mb{R}^{n+1}_+ \times [c, \iy)} \sigma_{s}^{ 1-2 \alpha}(t) x_{n+1}^{-a} \ \lvert \widetilde{\mc{H}_s} w \rvert ^2 \ G.$$ 
 \end{proof}

\section{Proof of the main results}\label{s:main}
Given the Carleman estimate in  Theorem \ref{carl1-un}, we now argue as in the proof of \cite[Lemma 5]{EKVP} to obtain the following $L^2$  lower bounds for the rescaled function $U_R$ in \eqref{u_r} which solves  \eqref{exprob-un-rep}.

\begin{theorem}
\label{estimates-for-resc}
Given $\tilde{\kappa}\in (0, 1],$ there exist large universal constant $M=M(n, s, \tilde{\kappa})$ and $\rho\in (0, 1)$ such that the following holds true:

If $U_R$ is as in \eqref{u_r} with $R^{(n+a+1)}\int_{\mathbb{B}_{\rho}^{+}} U_R^2(X, 0) x_{n+1}^{a}\, dX\geq \tilde{\kappa}$ ( note that this inequality in turn is assured by \eqref{lob}).  Then
 \begin{enumerate}
    \item 
For sufficiently small $\epsilon>0$ and $R \geq M$ we have     
\begin{align}
\int_{\mathbb{B}_{2}^{+}} x_{n+1}^a \ U_R(X, 0)^2 \ e^{-\frac{|X|^2 R^2}{\epsilon}} \ dX\geq e^{-MR^2 \log\big(\frac{1}{\epsilon}\big)}.
\label{res-main-1}
\end{align}
 \item 
For all $0\leq r<\f{1}{2},$ we have 
\begin{align}
\int_{\mathbb B_{r}^+} U_R^2(X, 0) x_{n+1}^a dX\geq e^{-MR^2 \log\big(\frac{2}{r}\big)}.
\label{res-main-2}
\end{align}
 \end{enumerate}
\end{theorem}
\begin{proof}
Let us highlight the key steps in the proof. The key ingredients are  the quantitative  Carleman estimate in Theorem \ref{carl1-un} and the improved  monotonicity in time result in  Lemma~\ref{lemma-mont}.
\medskip

\noindent{\textbf{Step 1:}} Let  $f= \eta(t)\phi(X) U_R,$ where $\phi \in C_0^{\infty}(\mathbb B_3)$ is a spherically symmetric cutoff such that $0\le \phi\le 1$ and $\phi \equiv1$ on $\mathbb B_{2}.$ Moreover, let $\eta$ be a cutoff in time such that $\eta=1$ on $[0, \f{1}{8\lambda}]$ and supported in $[0, \f{1}{4\lambda}).$ Since $U_R$ solves \eqref{exprob}, we see that the function $f$ solves the problem
\begin{equation}\label{feq1}
\begin{cases}
x_{n+1}^a f_t + \operatorname{div}(x_{n+1}^a \nabla f) =\phi x_{n+1}^a U_R\eta_{t}+ 2 x_{n+1}^a \eta\langle\nabla U_R,\nabla \phi\rangle  + \eta\operatorname{div}(x_{n+1}^a \nabla \phi) U_R\ \ \ \ \ \ \ \text{in} \ \mathbb B_5^+\times [0, \f{1}{R^2}),
\\	
f((x,0,t)= U_R(x,0,t)\phi(x,0)\eta(t)
\\

				\py f(x,0, t)= R^{2s} V_{R}f\ \ \ \text{in}\ \ \ \ \  B_5\times [0, \f{1}{R^2})
\end{cases}
\end{equation}
Since $\phi$ is symmetric in the $x_{n+1}$ variable, we have $\phi_{n+1}\equiv 0$ on $\{x_{n+1}=0\}$. Since $\phi$ is smooth, the following estimates are true, see \cite[(3.31)]{ABDG}.
\begin{equation}\label{obs1}
\begin{cases}
\operatorname{supp} (\nabla \phi) \cap \{x_{n+1}>0\}  \subset \mathbb B_3^+ \setminus \mathbb B_{2}^+
\\
|\operatorname{div}(x_{n+1}^a \nabla \phi)| \leq C x_{n+1}^a\ \mathbf 1_{\mathbb B_3^+ \setminus \mathbb B_{2}^+}.
\end{cases}
\end{equation}

\medskip

\noindent{\textbf{Step 2:}} The  Carleman estimate \eqref{care1-un} applied to $f$ (more precisely, a shifted in time version of \eqref{care1-un}) yields the following inequality for sufficiently large $\alpha$ satisfying $\alpha\geq MR^2$ and $0 < c \le \f{1}{5\ld}$   
\begin{align}
\notag & \alpha^2 \int_{\mb{R}^{n+1}_+ \times [0, \iy)} x_{n+1}^{a} (\sigma_{s}(t+c))^{-2 \alpha} \ f^2 \ G(X, t+c)  + \alpha \int_{\mb{R}^{n+1}_+ \times [0, \iy)} x_{n+1}^{a} (\sigma_{s}(t+c))^{1-2 \alpha}\  |\nabla f|^2 \ G(X, t+c)  \\
			\notag & \lesssim  M\int_{\mb{R}^{n+1}_+ \times [0, \iy)} \sigma_{s}^{ 1-2 \alpha}(t+c) x_{n+1}^{-a} \ \lvert \phi x_{n+1}^a U_R\eta_{t}+ 2 x_{n+1}^a \eta\langle\nabla U_R,\nabla \phi\rangle  + \eta\operatorname{div}(x_{n+1}^a \nabla \phi) U_R \rvert ^2 \ G(X, t+c)\\
   &+ \sigma_{s}^{-2 \alpha}(c) \left\{ - \f{c}{M} \int_{t=0} x_{n+1}^a \ |\nabla f(X,0)|^2 \ G(X,c) \ dX + \alpha M\int_{t=0} x_{n+1}^a \ |f(X, 0)|^2 \ G(X,c) \ dX\right\}\notag\\
   \notag & \lesssim  M\lambda^2\int_{\mb{R}^{n+1}_+ \times [0, \iy)} (\sigma_{s}(t+c))^{1-2 \alpha}\, G(X, t+c)\, x_{n+1}^{a} |U_R|^2 \mathbf 1_{[\f{1}{8\lambda}, \f{1}{4\lambda})}\\
   \notag &+M\int_{\mb{R}^{n+1}_+ \times [0, \iy)}  x_{n+1}^a \sigma_{s}^{ 1-2 \alpha}(t+c)\{|\nabla U_R|^2+|U_R|^2\}\mathbf 1_{\mathbb B_3 \setminus \mathbb B_{2}} \eta^2 G(X, t+c)\\
   &+ \sigma_{s}^{-2 \alpha}(c) \left\{ - \f{c}{M} \int_{t=0} x_{n+1}^a \ |\nabla f(X,0)|^2 \ G(X,c) \ dX + \alpha M\int_{t=0} x_{n+1}^a \ |f(X, 0)|^2 \ G(X,c) \ dX\right\}.\label{req1-un}
\end{align}

\noindent{\textbf{Step 3:}} Now we plug the following estimate ( see \cite[(4.24)]{AA}) 
\begin{align}\label{step3-un}\sigma_{s}^{ 1-2 \alpha}(t+c)G(X, t+c)\lesssim M^{2\alpha-1} \lambda^{2\alpha+\f{n+a+1}{2}}, \ (X,t) \in \mb{B}_3^+ \times [0, 1/4\lambda) \setminus \mb{B}_2^+ \times [0, 1/8 \lambda)\end{align}
in \eqref{req1-un} yielding
\begin{align}
\notag & \alpha^2 \int_{\mb{R}^{n+1}_+ \times [0, \iy)} x_{n+1}^{a} (\sigma_{s}(t+c))^{-2 \alpha} \ f^2 \ G(X, t+c) + \alpha \int_{\mb{R}^{n+1}_+ \times [0, \iy)} x_{n+1}^{a} (\sigma_{s}(t+c))^{1-2 \alpha}\  |\nabla f|^2 \ G(X, t+c)  \\ 
\notag & \lesssim   M^{2\alpha+\f{n+a+1}{2}} \alpha^{2\alpha+\f{n+a+1}{2}}\int_{[0, \f{1}{4\lambda})}\int_{\mathbb B_{3}^+} x_{n+1}^{a} \{|\nabla U_R|^2+|U_R|^2\}\\
   &+ \sigma_{s}^{-2 \alpha}(c) \left\{ - \f{c}{M} \int_{t=0} x_{n+1}^a \ |\nabla f(X,0)|^2 \ G(X,c) \ dX + \alpha M\int_{t=0} x_{n+1}^a \ |f(X, 0)|^2 \ G(X,c) \ dX\right\}\notag\\
   \notag & \lesssim  M^{2\alpha+\f{n+a+1}{2}} \alpha^{2\alpha+\f{n+a+1}{2}}R^{4} \ \text{(using   Lemma \ref{reg1} for $U$ which implies the derivative bounds for $U_R$)}\\
   &+ \sigma_{s}^{-2 \alpha}(c) \left\{ - \f{c}{M} \int_{t=0} x_{n+1}^a \ |\nabla f(X,0)|^2 \ G(X,c) \ dX + \alpha M\int_{t=0} x_{n+1}^a \ |f(X, 0)|^2 \ G(X,c) \ dX\right\}.\label{step4-un}
\end{align}

\noindent{\textbf{Step 5:}} Since $\phi=1$ on $\mathbb{B}_{2}$ and $\eta=1$ on $[0, \f{1}{8\lambda})$,  for small enough $\rho<\f{1}{2},$ which will be chosen later and $0 < c \leq \f{\rho^2}{8\lambda}$, we obtain  
\begin{align}
\notag & \alpha^2 \int_{\mb{R}^{n+1}_+ \times [0, \iy)} x_{n+1}^{a} \sigma_{s}^{-2 \alpha}(t+c) \ f^2 \ G(X, t+c)\\
&\geq  \notag \alpha^2 \int_{[0, \f{1}{8\lambda})}\int_{\mathbb{B}_{2}^+}   \sigma_{s}^{-2 \alpha}(t+c) \ x_{n+1}^{a} U_R^2 \ G(X, t+c)\\
 &\geq \notag \alpha^2 \int_{[0, \f{\rho^2}{4\lambda})}\int_{\mathbb{B}_{2\rho}^{+}} \sigma_{s}^{-2 \alpha}(t+c) \  (t+c)^{-\f{n+a+1}{2}} e^{-\f{|X|^2}{4(t+c)}} x_{n+1}^{a} U_R^2\\
 \notag &\geq \alpha^2 \int_{[0, \f{\rho^2}{4\lambda})} (t+c)^{-2\alpha} (t+c)^{-\f{n+a+1}{2}} e^{-\f{\rho^2}{(t+c)}}\int_{\mathbb{B}_{2\rho}^{+}}  x_{n+1}^{a}U_R^2\\
 &\geq \alpha^2 \int_{[c, c+\f{\rho^2}{4\lambda})} t^{-2\alpha} t^{-\f{n+a+1}{2}} e^{-\f{\rho^2}{t}} \int_{\mathbb{B}_{\rho}^{+}}  x_{n+1}^{a}U_R^2(X, 0)\notag\\
& \geq \alpha^2 \f{1}{M} \int_{[\f{\rho^2}{8\lambda}, \f{\rho^2}{4\lambda})} t^{-2\alpha} t^{-\f{n+a+1}{2}} e^{-\f{\rho^2}{t}} R^{-(n+a+1)}\notag\ \ (\text{using}\ \eqref{mon-un-1} )\\
&\geq \alpha^2 \f{1}{M} \left(\f{\rho^2}{4\lambda}\right)^{-\big(2\alpha+\f{n+a+1}{2}\big)} e^{-8\lambda} \left(\frac{\rho^2}{4\lambda}\right)
R^{-(n+a+1)}\notag\\
&\geq \frac{\delta^2 4^{2\alpha+\f{n+a+1}{2}}\lambda^{2\alpha+\f{n+a+1}{2}+1}}{8 M}(e^{4/\delta^2}\rho^2)^{-2\alpha}\rho^{2-{(n+a+1)}}R^{-(n+a+1)}.\notag
\end{align}

\noindent{\textbf{Step 6:}} The above computation and \eqref{step4-un} implies that 
\begin{align}
&\frac{\delta^2 4^{2\alpha+\f{n+a+1}{2}}\lambda^{2\alpha+\f{n+a+1}{2}+1}}{8 M}(e^{4/\delta^2}\rho^2)^{-2\alpha}\rho^{2-{(n+a+1)}}R^{-(n+a+1)}\\
&\notag  \lesssim  M^{2\alpha+\f{n+a+1}{2}} \alpha^{2\alpha+\f{n+a+1}{2}} R^{4} \\
   &+ \sigma_{s}^{-2 \alpha}(c) \left\{ - \f{c}{M} \int_{t=0} x_{n+1}^a \ |\nabla f(X,0)|^2 \ G(X,c) \ dX + \alpha M\int_{t=0} x_{n+1}^a \ |f(X, 0)|^2 \ G(X,c) \ dX\right\}.\label{step6}
\end{align}
To absorb the first term in the right hand side into the left, we need
\begin{align}
\notag &\frac{\delta^2 4^{2\alpha+\f{n+a+1}{2}}\lambda^{2\alpha+\f{n+a+1}{2}+1}}{8 M}(e^{4/\delta^2}\rho^2)^{-2\alpha}\rho^{-\f{n+a+1}{2}}\geq 8M^{2\alpha+\f{n+a+1}{2}} \alpha^{2\alpha+\f{n+a+1}{2}}R^{4} R^{n+a+1}
\end{align}
In view of the fact that $\alpha\sim R^2$ the above will be guaranteed if we choose $\rho$ such that
\begin{align}
\notag& \frac{\delta^2 4^{2\alpha+\f{n+a+1}{2}}\lambda^{2\alpha+\f{n+a+1}{2}+1}}{8 M}(e^{4/\delta^2}\rho^2)^{-2\alpha}\rho^{-\f{n+a+1}{2}}\geq 8 M^{2\alpha+\f{n+a+1}{2}} \lambda^{2\alpha+\f{n+a+1}{2}} R^{4} R^{n+a+1}\ (\text{since}\ \lambda \geq \alpha)\\
\notag&\impliedby  \delta^24^{2\alpha+\f{n+a+1}{2}}(e^{4/\delta^2}M\rho^2)^{-2\alpha}\geq 64 M^{\f{n+a+1}{2}+1}  R^{4} R^{n+a+1}\\
\notag&\impliedby (e^{4/\delta^2}M\rho^2)^{-2\alpha}\geq M^{\f{n+a+1}{2}+1} \,  \  \ (\text{as}  \  \ \delta^2 4^{2\alpha+\f{n+a+1}{2}}\geq 64 R^{4} R^{n+a+1}\,\text{considering}\,\,\alpha\sim R^2),
\end{align}
the above inequality will be true provided \begin{equation}\label{choicer} e^{4/\delta^2}M \rho^2\leq \f{1}{16}.\end{equation} Therefore, for $\alpha\geq M R^2$, \eqref{step6}  and the fact that $\sigma_s(c) \geq c e^{-N}$ implies  that \begin{align}
& \alpha^{2\alpha} e^{-2N\alpha} c^{2\alpha} \lesssim \alpha M\int_{t=0} x_{n+1}^a \ |f(X, 0)|^2 \ G(X,c) \ dX.
\end{align}
Now letting  $\alpha \geq M R^2$ with  $M >> e^{2N}$, we now  put  $c = \frac{\epsilon}{4R^2}$ where $\epsilon \leq \frac{\rho^2 \delta^2}{2M}$ and   consequently obtain from above 
\begin{align*}
\int_{\mathbb{B}_{2}^{+}} x_{n+1}^a \ U_R(X, 0)^2 \ e^{-\frac{|X|^2 R^2}{\epsilon}} \ dX\geq e^{-MR^2 \log\big(\frac{1}{\epsilon}\big)}.
\end{align*}
This finishes the proof of (1).

We now proceed with the proof of (2).

For the above mentioned choice of $\rho$  as in \eqref{choicer} and  by taking large $\alpha,$ \eqref{step6} implies for $c \leq \f{\rho^2}{8 \lambda} \sim \f{\epsilon}{R^2}$ that the following inequality holds
\begin{align}
\label{help}&\f{c}{M} \int x_{n+1}^a \ |\nabla f(X,0)|^2 \ G(X,c) \ dX\leq  R^2 M\int x_{n+1}^a \ |f(X, 0)|^2 \ G(X,c) \ dX\\
\implies& 2c\int  x_{n+1}^a \ |\nabla f(X,0)|^2 e^{-\f{|X|^2}{4c}}+\f{n+a+1}{2}\int x_{n+1}^a \ |f(X, 0)|^2 e^{-\f{|X|^2}{4c}} dX\notag \\
&\leq M^3 R^2 \int_{t=0} x_{n+1}^a \ |f(X, 0)|^2 e^{-\f{|X|^2}{4c}} dX.\label{stepfinal}
\end{align}
At this point \eqref{stepfinal} combined with Lemma~\ref{do} allow us to infer for a new $M$  that the following doubling inequality holds
\begin{align}
\label{adb}
	\int_{\mathbb B_{2r}^+} U_R^2(X, 0) x_{n+1}^a dX \leq e^{M R^2} \int_{\mathbb B_{r}^+} U_R^2(X, 0) x_{n+1}^a dX
\end{align}
for all $0\leq r<\f{1}{2}.$ Now given $r \leq 1/2$, choose $k \in \mathbb N$ such that $2^{-k} \leq r \leq 2^{-k+1}$. Iterating the above doubling inequality when $r= 2^{-j}$ with $j =0,\dots, k-1$ we obtain
\begin{equation}\label{db1}
\int_{\mathbb B_1^+} U_R^2(X,0) x_{n+1}^a dX \leq e^{2M R^2 \log(1/r)} \int_{\mathbb B_r^+} U_R^2(X,0) x_{n+1}^a dX.
\end{equation}
The conclusion follows from \eqref{db1}  with a new $M$ by noting that
\[
\int_{\mathbb B_1^+} U_R^2(X,0) x_{n+1}^a dX \geq \int_{\mathbb B_\rho^+} U_R^2(X,0) x_{n+1}^a dX \geq R^{-(n+1+a)} \tilde \kappa.
\]

\end{proof}
From Theorem \ref{estimates-for-resc}, we obtain the following decay estimates at infinity for the solution  $U$ to \eqref{exprob}.
\begin{theorem}
\label{main-thin}
Let $U$ be a solution of the original problem \eqref{exprob}.  
\begin{enumerate}
    \item  There exists a universal large constant $M$ such that for all $x_{0}\in \mathbb{R}^n$ with $|x_{0}|\geq M$ we have
    \begin{align}
      \label{org-main-1}\int_{\mathbb{B}^+_{|x_0|/2}((x_0,0))} U^2(X, \tilde t) x_{n+1}^a dX\geq e^{-M|x_0|^2},
    \end{align}
   for $\tilde t \in [t_0, t_0 + \tilde \delta)$ where $t_0$ is as in Lemma \ref{com-conse} and $\tilde \delta$ is as in Lemma \ref{com-conse1}. 
\item 
Also we have
\begin{align}
\label{org-main-2}
\int_{\mathbb{B}^+_{1}((x_0,0))} U^2(X, \tilde t) x_{n+1}^a dX\geq e^{-M|x_0|^2\log(|x_0|)}, \ \tilde t \in [t_0, t_0 +\tilde \delta).   
\end{align}
\end{enumerate}
\end{theorem}

\begin{proof}[Proof of part (1)]
Under the hypothesis of Theorem \ref{main} ( with $c=1$), we have from Lemma \ref{com-conse1} that for $\tilde t \in [t_0, t_0 +\tilde \delta)$
\begin{align}
\label{lower}
 \int_{\mathbb B_{1}^{+}} x_{n+1}^{a}U^2(X, \tilde t) dX\geq \tilde{\kappa}.
\end{align}
Now let $\rho$ be the number associated to $\tilde{\kappa}$ as in Theorem~\ref{estimates-for-resc}.  For  for each $\tilde t$ and $x_0$ such that   $|x_{0}| \geq M$, let $R:=2|x_{0}|/\rho$ and $U_R$ be as in \eqref{u_r}. From \eqref{lob} we have \begin{align}
R^{(n+a+1)}\int_{\mathbb{B}_{\rho}^{+}} U_{R}^2(X, 0) x_{n+1}^{a}\, dX=\int_{\mathbb{B}_{2|x_0|}^{+}((x_{0}, 0))} U^2(X, \tilde t) x_{n+1}^{a}\, dX\geq\int_{\mathbb B_{1}^{+}} x_{n+1}^{a}U^2(X, \tilde t) dX\geq \tilde{\kappa}.
\end{align} 
Thus $U_R$ satisfies the hypothesis in Theorem~\ref{estimates-for-resc}. Hence for small $\epsilon>0$ we have
\begin{align}
\notag &\int_{\mathbb{B}_{2}^{+}} x_{n+1}^a \ U_{R}(X, 0)^2 \ e^{-\frac{|X|^2 R^2}{\epsilon}} \ dX\geq e^{-MR^2 \log\big(\frac{1}{\epsilon}\big)}\\
\notag &\iff \int_{\mathbb{B}_{2R}^{+}((x_{0},0))} x_{n+1}^a \ U(X, \tilde t)^2 \ e^{-\frac{|X-(x_0, 0)|^2}{\epsilon}} \ dX\geq R^{n+a+1} e^{-MR^2 \log\big(\frac{1}{\epsilon}\big)}\\
\notag &\implies \left(\int_{\mathbb{B}_{|x_{0}|/2}^{+}((x_{0},0))} .. +\int_{\mathbb{B}_{4|x_{0}|/\rho}^{+}(x_{0})\setminus\mathbb{B}_{|x_{0}|/2}^{+}(x_{0})} .. \right) \geq R^{n+a+1} e^{-MR^2 \log\big(\frac{1}{\epsilon}\big)}\\
\label{int1}&\implies \int_{\mathbb{B}_{|x_{0}|/2}^{+}((x_{0},0))} x_{n+1}^a \ U(X, \tilde t)^2\,dX+CR^{n+a+1} e^{-R^2 \rho^2/16\epsilon}\geq R^{n+a+1} e^{-MR^2 \log\big(\frac{1}{\epsilon}\big)},
\end{align}
where we have used the fact that $\|U\|_{L^{\infty}}\leq C$ to bound the integral
\[
\int_{\mathbb{B}_{4|x_{0}|/\rho}^{+}(x_{0})\setminus\mathbb{B}_{|x_{0}|/2}^{+}(x_{0})} .. .\]
in \eqref{int1} above.
 Now if  $\epsilon>0$ is chosen  sufficiently small,  then the term  $CR^{n+a+1} e^{-R^2 \rho^2/16\epsilon}$ can be absorbed in the right hand side of \eqref{int1}.  Consequently,  we  can conclude that for a new $M$ (depending also on $\epsilon$) the following estimate holds
$$\int_{\mathbb{B}_{|x_{0}|/2}^{+}((x_{0},0))} x_{n+1}^a \ U(X, \tilde t)^2\,dX \geq e^{-MR^2}.$$
This completes the proof of \eqref{org-main-1}.

\medskip

To prove \eqref{org-main-2}, we apply \eqref{res-main-2} to the function $U_{R}$ at the scale $r=\frac{1}{R},$ which yields
\begin{align}
\notag &\int_{\mathbb B_{1/R}^+} U_{R}^2(X, 0) x_{n+1}^a dX\geq e^{-MR^2 \log(2R))}\\
\notag \implies &\int_{\mathbb B_{1}^+((x_{0},0))} U^2(X, \tilde t)\, x_{n+1}^a dX\geq R^{n+a+1}e^{-MR^2 \log(2R)}\\
&\implies\int_{\mathbb B_{1}^+((x_{0},0))} U^2(X, \tilde t)\, x_{n+1}^a dX\geq e^{-MR^2 \log(2R)},
\end{align}
since $R \geq 1$. The conclusion thus follows with a larger $M$ by noting that $|x_0| \sim R$ once $\rho$ gets fixed as in Theorem \ref{estimates-for-resc}.
\end{proof}

As a direct consequence of the estimate \eqref{org-main-2} in Theorem \ref{main-thin}, the  following asymptotic decay estimate holds  in space-time regions for the solution  $U$  to the extension problem  \eqref{exprob}. \begin{theorem}
\label{main-thick}  
Under the assumption of Theorem~\ref{main-thin}, there exist universal constants $M$ and $ \tilde \delta\in (0, 1)$ such that
for  $|x_0|\geq M$ we have
\begin{equation}\label{lw2}
\int_{\mathbb{B}^+_{1}((x_0,0))\times [t_0+ \tilde \delta/2,  t_{0}+3\tilde \delta/4)} U^2(X, t) x_{n+1}^a dX\,dt\geq e^{-M|x_0|^2\log(|x_0|)}.
\end{equation}
where $t_0$ is as in Lemma \ref{com-conse}.

\end{theorem}

\subsection{Propagation of smallness and the Proof of Theorem \ref{main}}
We now transfer the decay estimate at the bulk as in Theorem \ref{main-thick}  to  the boundary via an appropriate propagation of smallness estimate derived in \cite[Corollary 4.4]{AryaBan} using which Theorem \ref{main} follows. 

\begin{proof}[Proof of Theorem \ref{main}]
We first note that from the hypothesis of Theorem \ref{main} ( recall  that we are assuming $c=1$),  we infer that   the estimate \eqref{lw2}  in Theorem \ref{main-thick} holds.   We now use the following  variant of the propagation of smallness estimate as derived in \cite[Corollary 4.4]{AryaBan}.
\begin{align}\label{pos}
&||x_{n+1}^{a/2} U||_{L^{2}(\mathbb{B}^+_{1}((x_0,0))\times [t_0+ \tilde \delta/2, t_{0}+3\tilde \delta/4))}\\& \leq C||u||_{L^{2}(\mathbb R^{n+1})}^{1-\theta} \left( ||Vu||_{L^{2}(B_{3/2} (x_0) \times [t_0 + \tilde \delta/4, t_0 + 5 \tilde \delta /6))}^{\theta} + ||u||_{W^{2,2}(B_{3/2} (x_0) \times [t_0 + \tilde \delta/4, t_0 + 5 \tilde \delta /6))}^{\theta} \right)\notag\\
& + C\left( ||Vu||_{L^{2}(B_{3/2} (x_0) \times [t_0 + \tilde \delta/4, t_0 + 5 \tilde \delta /6))} + ||u||_{W^{2,2}(B_{3/2} (x_0) \times [t_0 + \tilde \delta/4, t_0 + 5 \tilde \delta /6))} \right), \notag
\end{align}
where $\theta \in (0,1)$ is universal and  $$ 	||u||_{W^{2,2}}\overset{def}=||u||_{L^2} + ||\nabla_x u||_{L^2}+ ||\nabla^2_x u||_{L^2}+ ||u_t||_{L^2}.$$ 
Note that \eqref{pos} follows from \cite[Corollary 4.4]{AryaBan} by  a translation in space and a standard covering argument. Note that in view of \eqref{vasump}, the right hand side of \eqref{pos} is upper bounded by 
$$C\left( ||u||_{W^{2,2}(B_{3/2} (x_0) \times [t_0 + \tilde \delta/4, t_0 + 5 \tilde \delta /6))}+ ||u||_{W^{2,2}(B_{3/2} (x_0) \times [t_0 + \tilde \delta/4, t_0 + 5 \tilde \delta /6))}^{\theta}\right).$$   Now since we are interested in a lower bound, so without loss of generality we may assume that $$||u||_{W^{2,2}(B_{3/2} (x_0) \times [t_0 + \tilde \delta/4, t_0 + 5 \tilde \delta /6))} \leq 1.$$
Using this along with \eqref{lw2}, we obtain that the following inequality holds for a larger $M$ universal and $|x_0| \geq M$
\begin{equation}\label{lw5}
||u||_{W^{2,2}(B_{3/2} (x_0) \times [t_0 + \tilde \delta/4, t_0 + 5 \tilde \delta /6))}  \geq e^{-M|x_0|^2\log(|x_0|)}.
\end{equation}
In order to get an $L^{2}$ decay as claimed in Theorem \ref{main},  we now make use of the interpolation type inequalities in Lemma \ref{interpol}. Let $\phi$ be a smooth function supported in $\mb{B}_{7/4} ((x_0,0)) \times  (t_0 + \tilde \delta/8, t_0 +11 \tilde /12)$ such that $\phi \equiv 1$ in $\mb{B}_{3/2}((x_0,0)) \times  [t_0 + \tilde \delta/4, t_0 + 5 \tilde \delta /6))$. Then by applying \eqref{inter1} to $f$ we get also by using the regularity estimates in Lemma \ref{reg1} that the following holds for any $\eta_1 \in (0,1)$
\begin{align}\label{ino1}
& ||\nabla_x u||_{L^{2}(B_{3/2}(x_0) \times [t_0 + \tilde \delta/4, t_0 + 5 \tilde \delta /6)) }\leq   ||\nabla_x f||_{L^{2}(\mathbb R^{n+1})}\\& \leq  C \eta_1^s ||x_{n+1}^{a/2} U||_{L^2(\mb{B}_{7/4}^+ ((x_0, 0)) \times (t_0 + \tilde \delta/8, t_0 +11 \tilde \delta/12))} + C \eta_1^{-1} ||u||_{L^2(B_{7/4} (x_0) \times (t_0 + \tilde \delta/8, t_0 +11 \tilde \delta /12))}.\notag\end{align}
Similarly, by applying \eqref{inter1} to $\nabla_x f$ and by using  the second derivative estimates in Lemma \ref{reg1} we get for any $\eta \in (0,1)$

\begin{equation}\label{ino2}
||\nabla_x^2 u||_{L^{2}(B_{3/2}(x_0) \times [t_0 + \tilde \delta/4, t_0 + 5 \tilde \delta /6)) } \leq  C \eta^s ||x_{n+1}^{a/2} U||_{L^2(\mb{B}_{2}^+ ((x_0, 0)) \times (t_0 + \tilde \delta/16, t_0 + \tilde \delta))} + C \eta^{-1} ||\nabla_x f||_{L^{2}(\mathbb R^{n+1})} .\end{equation}

Then using \eqref{ino1} in \eqref{ino2}, we thus obtain

\begin{align}\label{ino4}
&||\nabla_x^2 u||_{L^{2}(B_{3/2}(x_0) \times [t_0 + \tilde \delta/4, t_0 + 5 \tilde \delta /6)) }  \leq C\eta^s ||x_{n+1}^{a/2} U||_{L^2(\mb{B}_{2} ^+ ((x_0, 0))\times (t_0 + \tilde \delta/16, t_0 + \tilde \delta))}\\
& + C \eta^{-1} \eta_1^s ||x_{n+1}^{a/2} U||_{L^2(\mb{B}_{7/4}^+((x_0, 0)) \times (t_0 + \tilde \delta/8, t_0 +11 \tilde \delta/12))} + C(\eta \eta_1)^{-1} ||u||_{L^2(B_{7/4}  (x_0) \times (t_0 + \tilde \delta/8, t_0 +11 \tilde \delta /12))}. \notag \end{align}
We now take $\eta_1=\eta^3$. This ensures that
\begin{equation}\label{eta}
\eta^{-1} \eta_1^s= \eta^{3s-1} \leq \eta^{s}\ \text{as $s\geq 1/2$ and $\eta<1$}.
\end{equation}

 Substituting this value of $\eta_1$  in \eqref{ino4}, using \eqref{eta}  and also by using Lemma \ref{reg} we find
\begin{equation}\label{ino5}
||\nabla_x^2 u||_{L^{2}(B_{3/2}(x_0) \times [t_0 + \tilde \delta/4, t_0 + 5 \tilde \delta /6)) } \leq C \eta^s + C \eta^{-4} ||u||_{L^2(B_2(x_0) \times [t_0 +t_0 + \tilde \delta))}.
\end{equation}
Similarly  by applying \eqref{inter2} to $f$ and by using the estimates in Lemma \ref{reg1} and Lemma \ref{reg} we find
\begin{equation}\label{ino8}
||u_t||_{L^{2}(B_{3/2}(x_0) \times [t_0 + \tilde \delta/4, t_0 + 5 \tilde \delta /6)) } \leq C \eta^s + C \eta^{-4} ||u||_{L^2(B_2(x_0) \times [t_0 +t_0 + \tilde \delta))}.\end{equation}
Thus from \eqref{ino1}, \eqref{ino5} and \eqref{ino8} it follows that
\begin{equation}\label{gt1}
||u||_{W^{2,2}(B_{3/2} (x_0) \times [t_0 + \tilde \delta/4, t_0 + 5 \tilde \delta /6))} \leq C\eta^s  + C\eta^{-4}||u||_{L^{2}(B_2(x_0) \times  [t_0, t_0 +\tilde \delta))}.
\end{equation}
Now using \eqref{lw5}, we deduce from \eqref{gt1} that the following inequality holds for $|x_0| \geq M$,
\begin{equation}\label{gt2}
e^{-M|x_0|^2\log(|x_0|)} \leq C\eta^s + C\eta^{-4}||u||_{L^{2}(B_2(x_0) \times  [t_0, t_0 +\tilde \delta))}.
\end{equation}
Now by letting \begin{equation}\label{choice1}\eta^s= \frac{e^{-M|x_0|^2\log(|x_0|)}}{2C},\end{equation} we find  that the  first term  on the right hand side in \eqref{gt2} can  be absorbed in the left hand side and we consequently obtain for a new $M$
\begin{equation}\label{gt4}
\frac{\eta^4 e^{-M|x_0|^2\log(|x_0|)} }{2C} \leq ||u||_{L^{2}(B_2(x_0) \times  [t_0, t_0 +\tilde \delta))}. \end{equation}
Now by noting that in view of \eqref{choice1},  we have that $$\eta^4 \sim e^{-\frac{4M}{s}|x_0|^2\log(|x_0|)}.$$ Therefore, by using this in \eqref{gt4}, we find that the conclusion follows with a new $M$  by also using that

\[
||u||_{L^{2}(B_2(x_0) \times  [0, 1))} \geq ||u||_{L^{2}(B_2(x_0) \times  [t_0, t_0 +\tilde \delta))}.\]   This finishes the proof of Theorem \ref{main}  by noting that we have assumed  $c=1$ in Theorem \ref{main} (for the sake of simpler exposition of the ideas) and also  by observing that we are working with the backward version of the problem as in \eqref{exprob}.

\end{proof}
We now use the estimate in Theorem \ref{main} to finish the proof of the Landis-Oleinik type result in Corollary \ref{main0}.

\begin{proof}[Proof of Corollary \ref{main0}]
We show that
\begin{equation}\label{ef1}
 ||u||_{L^{2}(B_{1/2} \times (-1/4, 0])} =0. \end{equation}
 Suppose instead that is not the case and we have
 \begin{equation}\label{ef2}
 ||u||_{L^{2}(B_{1/2} \times (-1/4, 0])} \geq \theta>0.
 \end{equation}  
 Then by applying Theorem \ref{main} corresponding to this $\theta$, there exists some $M=M(\theta)$ such that
 \begin{equation}\label{ef4}
\int_{B_2(x_0) \times (-1, 0)} u^2(x, t)\,dx\,dt  \geq e^{-M |x_0|^2 \log |x_0|}\ \ \ \text{holds for all}\,\, |x_0| \geq M.
\end{equation} 
Now on the other hand, the hypothesis (assuming $c=1$) implies that   \begin{equation}\label{so1} \int_{-1}^0 u^2(x,t)\,dt \leq C e^{-|x|^{2+\epsilon}}\ \ \ \text{for all}\ x \in \mathbb R^n. \end{equation}
Therefore, by integrating  \eqref{so1} over the region $B_2(x_0)$ for $|x_0| \geq M$ with $M$ as in Theorem \ref{main} ( corresponding to the $\theta$ in \eqref{ef2})  we find  for a new $C$  that the following holds
\begin{equation}\label{so2}
\int_{B_2(x_0) \times (-1, 0)} u^2(x, t)\,dxdt \leq  C e^{-  \frac{|x_0|^{2+\epsilon}}{2^{2+\epsilon}  } }, \end{equation}
where we have used that for $x \in B_2(x_0)$, $|x| \geq \frac{|x_0|}{2}$ which can be ensured for $M>4$. This clearly contradicts \eqref{ef4} for large $|x_0|$ as 
\[
e^{-M |x_0|^2 \log |x_0|}\gg e^{-  \frac{|x_0|^{2+\epsilon}}{2^{2+\epsilon}  } },\]
as $|x_0|=R \to \infty$.  Thus \eqref{ef1} holds. So in particular, we have that $u$ vanishes to infinite order in space-time at $(0,0)$. Now we can apply the backward uniqueness result in \cite[Theorem 1.2]{BG}  to conclude that $u \equiv 0$ in $\mathbb R^{n} \times [-T, 0]$. 
\end{proof}

\begin{remark}
In the case when  the nonlocal equation \eqref{e0} holds for $t>0$, then we can also conclude that $u(\cdot, t) =0$ for $t>0$ by invoking  the forward uniqueness result in \cite{BG2}.
\end{remark}

\end{document}